\documentclass[oneside]{amsart}

\usepackage[a4paper]{geometry}
\usepackage[T1]{fontenc}
\usepackage[expert]{mathdesign}
\usepackage{lmodern}

\usepackage{url}					
\usepackage[english]{babel}
\usepackage[numbers,sort]{natbib}
\usepackage{amsmath}
\usepackage[inline]{enumitem}
\usepackage{wasysym}
\usepackage{amsthm}
\usepackage{amssymb}
\usepackage{latexsym}
\usepackage{amsaddr}
\usepackage{thmtools}
\usepackage{mathtools}
\usepackage{stmaryrd}
\usepackage{microtype}
\usepackage{caption}
\usepackage{cleveref}
\usepackage{setspace}

\declaretheoremstyle[
  headfont=\scshape,
  notefont=\normalfont, notebraces={(}{)},
  bodyfont=\itshape,
  postheadspace=.5em,
  spaceabove=9pt,
  spacebelow=6pt,
  headpunct={.}
]{thstyle}

\declaretheoremstyle[
  headfont=\scshape,
  notefont=\normalfont, notebraces={(}{)},
  bodyfont=\normalfont,
  spaceabove=9pt,
  spacebelow=6pt,
  postheadspace=.5em,
  headpunct={.}
]{defstyle}

\declaretheoremstyle[
  headfont=\itshape,
  notefont=\normalfont, notebraces={(}{)},
  bodyfont=\normalfont,
  spaceabove=6pt,
  spacebelow=6pt,
  postheadspace=.5em,
  headpunct={.}
]{remarkstyle}

\declaretheoremstyle[
  headfont=\scshape,
  notefont=\normalfont, notebraces={(}{)},
  bodyfont=\normalfont,
  spaceabove=9pt,
  spacebelow=6pt,
  postheadspace=.5em,
  headpunct={.}
]{exstyle}

\declaretheoremstyle[
  headfont=\itshape,
  notefont=\itshape, notebraces={(}{)},
  bodyfont=\normalfont,
  spaceabove=6pt,
  spacebelow=9pt,
  postheadspace=\labelsep,
  headpunct={.}
]{proofstyle}

\declaretheorem[style=proofstyle,unnumbered,qed=$\square$]{proof}
\declaretheorem[style=remarkstyle,numbered=no]{remark}

\declaretheorem[style=thstyle,numberwithin=section,refname={theorem,theorems},Refname={Theorem,Theorems}]{theorem}
\declaretheorem[style=thstyle,sibling=theorem,refname={proposition,propositions},Refname={Proposition,Propositions}]{proposition}
\declaretheorem[style=thstyle,sibling=theorem,refname={lemma,lemmas},Refname={Lemma,Lemmas}]{lemma}
\declaretheorem[style=thstyle,sibling=theorem,refname={corollary,corollaries},Refname={Corollary,Corollaries}]{corollary}
\declaretheorem[style=defstyle,sibling=theorem,refname={definition,definitions},Refname={Definition,Definitions}]{definition}

\newcommand{\verum}{\ensuremath{\top}}
\newcommand{\falsum}{\ensuremath{\bot}}
\newcommand{\limpl}{\ensuremath{\rightarrow}}
\newcommand{\lequiv}{\ensuremath{\leftrightarrow}}
\providecommand{\land}{\ensuremath{\wedge}}
\providecommand{\lor}{\ensuremath{\vee}}
\newcommand{\lneg}{\ensuremath{\neg}}

\newcommand{\proves}{\ensuremath{\vdash}}
\newcommand{\nproves}{\ensuremath{\nvdash}}
\newcommand{\ca}[1]{\ensuremath{\mathcal{#1}}}
\newcommand{\fk}[1]{\ensuremath{\mathfrak{#1}}}
\newcommand{\msGLP}{\ensuremath{\mathsf{GLP}^\ast}}
\newcommand{\mGLP}[1]{\ensuremath{\mathsf{GLP}^\ast_{#1}}}

\newcommand{\msGLPS}{\ensuremath{\mathsf{GLPS}^\ast}}
\newcommand{\msJ}{\ensuremath{\mathsf{J}^\ast}}
\newcommand{\mJ}[1]{\ensuremath{\mathsf{J}^\ast_{#1}}}

\newcommand{\msRCo}{\ensuremath{\mathsf{RC}^\ast\!\omega}}
\newcommand{\RC}{\ensuremath{\mathsf{RC}}}

\newcommand{\GLP}{\ensuremath{\mathsf{GLP}}}
\newcommand{\GLV}{\ensuremath{\mathsf{GLV}}}
\newcommand{\IPC}{\ensuremath{\mathsf{IPC}}}
\newcommand{\J}{\ensuremath{\mathsf{J}}}
\newcommand{\GL}{\ensuremath{\mathsf{GL}}}

\newcommand{\PA}{\ensuremath{\mathsf{PA}}}

\newcommand{\ndia}[1]{\ensuremath{\langle #1 \rangle}}
\newcommand{\nbox}[1]{\ensuremath{[#1]}}
\newcommand{\sort}[1]{\ensuremath{|#1|}}
\newcommand{\val}[1]{\ensuremath{\llbracket #1 \rrbracket}}

\renewcommand{\emptyset}{\ensuremath{\varnothing}}
\newcommand{\defequiv}{\ensuremath{\mathrel{{\Longleftrightarrow}_{\mathit{df}}}}}

\newcommand{\gives}{\ensuremath{\Rightarrow}}

\newcommand{\code}[1]{\ensuremath{{\ulcorner #1 \urcorner}}}
\newcommand{\Conn}{\ensuremath{\underline{\mathsf{Con}}}}
\newcommand{\Con}{\ensuremath{\mathsf{Con}}}
\newcommand{\Prov}{\ensuremath{\mathsf{Prv}}}
\newcommand{\Prf}{\ensuremath{\mathsf{Prf}}}
\newcommand{\mand}{\ensuremath{\mathrel{\&}}}
\newcommand{\ol}[1]{\ensuremath{\overline{#1}}}
\newcommand{\ul}[1]{\ensuremath{\underline{#1}}}
\newcommand{\stdmodel}{\ensuremath{\mathbb{N}}}
\newcommand{\PSpace}{\textsc{PSpace}}
\newcommand{\Th}{\ensuremath{\mathsf{Th}}}

\newcommand{\Ax}{\ensuremath{\mathsf{Ax}}}
\newcommand{\True}{\ensuremath{\mathsf{True}}}

\makeatletter
\g@addto@macro\bfseries{\boldmath}
\makeatother

\makeatletter
\newcommand{\leqnomode}{\tagsleft@true}
\newcommand{\reqnomode}{\tagsleft@false}
\makeatother

\crefformat{equation}{{}#2(#1)#3}

\title{\bf A Many-Sorted Variant of Japaridze's Polymodal Provability
  Logic}

\author{Gerald Berger}
\thanks{Corresponding author's address: Gerald Berger, Insitute of Logic and Computation, TU Wien, Favoritenstra\ss e 9--11, 1040 Wien, Austria}
\address{Institute of Logic and Computation, TU Wien}
\email{gberger@dbai.tuwien.ac.at}
\author{Lev D.~Beklemishev}
\email{bekl@mi.ras.ru}
\address{Steklov Mathematical Institute, Russian Academy of Sciences \\ National Research University Higher School of Economics \\  Moscow M.V.~Lomonosov State University}
\author{Hans Tompits}
\email{tompits@kr.tuwien.ac.at}
\address{Institute for Logic and Computation, TU Wien}

\thanks{A version of this article has been published in the \emph{Logic Journal of the IGPL}, 26(5): 505--538 (2018).}

\makeatletter
\let\xx@thm\@thm
\AtBeginDocument{\let\@thm\xx@thm}
\makeatother

\begin{document}
\setlist[1]{itemsep=0pt,parsep=6pt,topsep=9pt,labelindent=\parindent}
\setlist[enumerate,1]{label={\textnormal{(\roman*)}}, ref=(\roman*), leftmargin=3em}
\setlist[enumerate,2]{label={\textnormal{(\arabic*)}}, ref=(\arabic*)}
\interfootnotelinepenalty=10000

\maketitle
\begin{abstract}
  We consider a many-sorted variant of Japaridze's polymodal provability
  logic $\GLP$. In this variant, which is denoted $\msGLP$,
  propositional variables are assigned sorts $\alpha \leq \omega$, where
  variables of finite sort $n < \omega$ are interpreted as
  $\Pi_{n+1}$-sentences of the arithmetical hierarchy, while those of
  sort $\omega$ range over arbitrary ones. We prove that $\msGLP$ is
  arithmetically complete with respect to this interpretation. Moreover,
  we relate $\msGLP$ to its one-sorted counterpart $\GLP$ and prove that
  the former inherits some well-known properties of the latter, like
  Craig interpolation and \PSpace\ decidability. We also study a
  positive variant of $\msGLP$ which allows for an even richer
  arithmetical interpretation---variables are permitted to range over
  theories rather than single sentences. This interpretation in
  turn allows the introduction of a modality that corresponds to the
  full uniform reflection principle. We show that our positive variant
  of $\msGLP$ is arithmetically complete.

\end{abstract}

\section{Introduction}

The polymodal provability logic $\GLP$, due to~\cite{japaridze:1988}, has
received considerable interest in the mathematical logic community. The
language of $\GLP$ features modalities $\ndia{n}$, for every $n \geq 0$,
that can be arithmetically interpreted as \emph{$n$-consistency}, i.e.,
the modal formula $\ndia{n}\varphi$ expresses under this interpretation
that $\varphi$ is consistent with the set of all true
$\Pi_n$-sentences. This particular interpretation steered interest in
$\GLP$ in mainstream proof theory: in \cite{beklemishev:2004}, the
second author of this paper showed how $\GLP$ can act as a framework in
order to canonically recover an ordinal notation system for Peano
arithmetic~($\PA$) and its fragments. Moreover, based on these notions,
he obtained a rather abstract version of Gentzen's consistency proof for
$\PA$ by transfinite induction up to $\varepsilon_0$ and he formulated a
combinatorial statement independent from $\PA$~\citep{beklemishev:2006}.

This proof-theoretic analysis is based on the notion of \emph{graded
  provability algebra}. Let $T$ be an extension of $\PA$. Recall the
concept of \emph{Lindenbaum algebra of $T$}: its elements are equivalence
classes of the relation
\begin{align*}
  \varphi \sim \psi\ \defequiv\ T \proves \varphi \lequiv \psi.
\end{align*}
Let $[\varphi]$ denote the equivalence class of $\varphi$ with respect
to $\sim$. The algebra $\ca{L}_T$ is equipped with the standard Boolean
connectives and the relation
\begin{align*}
  [\varphi] \leq [\psi]\ \defequiv\ T \proves \varphi \limpl \psi.
\end{align*}
This turns $\ca{L}_T$ into a Boolean algebra,
\emph{the Lindenbaum algebra of $T$}. Thus, logical notions are brought
into an algebraic setting. The maximal element $\verum$ and the minimal
element $\falsum$ of this algebra are, respectively, the classes of all
provable and all refutable sentences of $T$ and deductively closed
extensions of $T$ correspond to \emph{filters} of $\ca{L}_T$ (see
\citep{beklemishev:2005} for details).

Let $\ndia{n}_T$ be a $\Pi_{n+1}$-formula that formalizes the notion of
$n$-consistency in arithmetic (see, e.g.,~\citep{beklemishev:2005} for a
definition of $\ndia{n}_T$). The \emph{graded provability algebra}
$\ca{M}_T$ of $T$ is the algebra $\ca{L}_T$ extended by operators
$\ndia{n}_T$ defined on the elements of $\ca{L}_T$ by
\begin{align*}
  \ndia{n}_T \colon [\varphi] \longmapsto
  [\ndia{n}_T\varphi],\quad\text{for $n \geq 0$.}
\end{align*}
Terms in the language of $\ca{M}_T$ can be identified with polymodal
formulas. Furthermore, for each sound and axiomatizable extension $T$ of
$\PA$, Japaridze's arithmetical completeness theorem for $\GLP$ states
that
\begin{align*}
  \GLP \proves \varphi(\vec{p}) \iff \ca{M}_T \models \forall
  \vec{p}\,(\varphi(\vec{p}) = \verum),
\end{align*}
where $\vec{p}$ are all the propositional variables from
$\varphi(\vec{p})$. The algebra $\ca{M}_T$ carries an additional
structure in the form of a distinguished family of subsets
\begin{align*}
  P_0 \subset P_1 \subset \cdots \subset \ca{M}_T,
\end{align*}
where $P_n$ is defined by the class of $\Pi_{n+1}$-sentences of the
arithmetical hierarchy. This family of subsets is called a
\emph{stratification of $\ca{M}_T$}~\citep{beklemishev:2004}. Since
$\ndia{n}_T$ is a $\Pi_{n+1}$-formula, the operator $\ndia{n}_T$ maps
$\ca{M}_T$ to $P_n$. The presence of a stratification thus admits to
turn $\ca{M}_T$ into a many-sorted algebra in which variables of sort
$n$ range over arithmetical $\Pi_{n+1}$-sentences. The notion of sort
can be readily extended to capture all polymodal terms. It is the goal
of this paper to investigate a modal-logical counterpart to this
many-sorted algebra.

Let us briefly comment on the general motivations for this study.  One
of the (global and ambitious) goals of relating provability algebras to
the ordinal analysis of theories was to shed more light on the
well-known and basic conceptual problem of ``natural ordinal notations''
in proof theory (see, e.g.,~\citep{Fefbug, kreisel:1977}). We
would like to understand general criteria distinguishing well-behaved
ordinal notation systems suitable for proof-theoretic analysis from the
``wild'' ones, as in Kreisel's counterexamples~\cite{kreisel:1977}.

The approach of provability algebras is an attempt to recast core
proof-theoretic results in a more abstract, essentially algebraic,
language. This amounts to introducing structures that are, on the one
hand, directly related to strong, computationally universal formal
systems, such as Peano arithmetic and its extensions. On the other hand,
from these structures one should be able to recover ordinal notation
systems in a canonical way. In other words, we consider the natural
ordinal notations problem as the question of what kind of information is
required for us to be able to speak about proof-theoretic ordinal
notation systems in a canonical way.

Within such a project it seems necessary to ``pack'' all relevant
proof-theoretic information into a suitable algebraic framework---and
the simpler this framework is the better. Basic results in the proof
theory of arithmetic can be viewed as either proofs of reflection
schemas restricted to arithmetical complexity classes $\Pi_n$, or as
$\Pi_n$-conservativity relationships between certain systems. Thus, the
stratification of the provability algebra into levels of the
arithmetical hierarchy of formulas seems to be part of the data that
necessarily has to be represented within the sought algebraic
framework. (Let us stress that, for example, introducing quantifiers in
the style of cylindric algebras would be an overkill, as we would obtain
structures that are not ``tame''.)

For example, the so-called \emph{reduction property} of provability
algebras is a key result needed for the proof-theoretic analysis of
Peano arithmetic. The most natural statement of this property in~\citep{beklemishev:2004} becomes purely algebraic only if the
stratification is part of the considered algebraic structure.\footnote{See
also~\citep{Bek17} for some generalizations of the reduction property
that can be stated without references to sorts.}

The present paper considers the most direct approach to incorporating
the stratification into the syntactic framework where the propositional
variables are assigned ``rigid'' sorts (types), for every $n \geq 0$, and
are understood as ranging over the classes of arithmetical
$\Pi_{n+1}$-sentences. The corresponding many-sorted variant of $\GLP$
will be denoted by $\msGLP$. Substitution in this logic is required to
respect the sorts of variables.\footnote{Thus, strictly speaking, our
  treatment does not yield a logic in the usual sense, since it is not
  closed under unrestricted substitutions. However, we shall use this
  term without further concern.} Our main result is a Solovay-style
arithmetical completeness theorem for $\msGLP$, i.e., for any sound and
axiomatizable extension $T$ of $\PA$ we have
\begin{align*}
  \msGLP \proves \varphi(\vec{p}) \iff \ca{M}_T \models \forall
  \vec{p}\,(\varphi(\vec{p}) = \verum),
\end{align*}
where $\vec{p}$ are all propositional variables from
$\varphi(\vec{p})$ and a quantifier binding such a variable of sort
$n$ only ranges over elements from $P_n$. In particular, we show that
the principle of \emph{$\Sigma_{n+1}$-completeness},
\begin{align*}
  \ndia{n}_Tp \limpl p
\end{align*}
(where $p$ is of sort $n$), in addition to the postulates of $\GLP$,
suffices to obtain arithmetical completeness. We observe that most of
the arguments in the proof of arithmetical completeness of $\GLP$ also
work for the sorted language. Thus, having sorted variables does not
really lead to a more complicated arithmetically complete system than
$\GLP$ itself.

A similar system has been studied by~\citet{visser:1981,Vis84} who
introduced a $\Sigma_1$-provability logic of $\PA$, i.e., in his logic,
variables are arithmetically interpreted as $\Sigma_1$-sentences (see
also~\citep{boolos:1995,ardeshir:2015}). The interpretation of
propositional variables as $\Sigma_1$-sentences also plays an important
role in the study of intuitionistic provability logic and its
variable-free fragment; see~\citep{Vis02a}.

In~\citep{Dzh94} a more flexible, yet more complicated approach is
considered, where types corresponding to $\Sigma_n$- and to
$\Pi_n$-sentences, for all $n\geq 1$, are not rigid but can be defined
using the modalities $\Sigma_n\varphi$ arithmetizing the predicate
``\emph{$\varphi$ is $\PA$-equivalent to a $\Sigma_n$-sentence}''. This
logic, however, lacks the necessary modalities $\ndia{n}$, for all
$n>0$, representing the higher reflection principles. It might be
interesting to consider the extension of $\GLP$ by modalities
$\Sigma_n$---however, at this point, it is not clear whether this system
has substantial advantages compared to the one with rigid
types.

The remainder of the paper is organized as follows. After this
introductory section, we define basic notions
in~\Cref{sec:2}. In~\Cref{sec:3}, we prove the arithmetical completeness
theorem for $\msGLP$. We continue our exposition on $\msGLP$
in~\Cref{sec:4} by proving that deciding provability in $\msGLP$ is
complete for \PSpace\ and providing a natural many-sorted
truth-provability logic. Moreover, we show that $\msGLP$ admits Craig
interpolation and study variants of $\msGLP$ that restrict the sorts and
modalities we are allowed to use. In~\Cref{sec:many}, we study a
positive variant $\msRCo$ of $\msGLP$ whose corresponding one-sorted
counterpart has recently gained attraction in the provability logic
community. In this fragment, we restrict ourselves to certain positive
formulas which allow us to focus on more general arithmetical
interpretations---variables are permitted to range over arithmetical
theories rather than single sentences. This in turn allows the
introduction of an additional modality $\ndia{\omega}$ that corresponds
to the full uniform reflection principle which has no finite, yet
recursive axiomatization. We prove that $\msRCo$ is arithmetically
complete for this interpretation.  We conclude the paper
in~\Cref{sec:5}.

\section{Preliminaries}
\label{sec:2}

\subsection{The Logics $\GLP$, $\msGLP$, and $\msJ$}

The polymodal provability logic $\GLP$ is formulated in the language of
the propositional calculus (using the connectives $\verum$, $\lneg$, and
$\land$ as primitives), enriched by unary connectives
$\ndia{0},\ndia{1},\ndia{2},\ldots$, called \emph{modalities}. Using
these connectives, formulas are built inductively in the usual way. The
dual connectives $\nbox{n}$, for every $n \geq 0$, are abbreviations
where $\nbox{n}\varphi$ stands for
$\lneg\ndia{n}\lneg\varphi$. Moreover, we abbreviate the standard
Boolean connectives using $\verum$, $\lneg$, and $\land$ in the usual
manner. 

The logic $\GLP$ is axiomatized by the following axiom schemas
and rules:\footnote{Usually, $\GLP$ is axiomatized by using $\nbox{n}$
  instead of $\ndia{n}$. However, it is more convenient for our purposes
  to use $\ndia{n}$, since we focus on $\Pi_{n+1}$-axiomatized
  concepts. Note that $\GLP$ is closed under the \emph{necessitation
    rule}: if $\GLP \proves \varphi$ then
  $\GLP \proves \nbox{n}\varphi$.}

\begin{enumerate}[label={(\roman*)}]
\item all tautologies of classical propositional logic;\label{glp:scheme:taut}
\item
  $\ndia{n}(\varphi \lor \psi) \limpl \ndia{n}\varphi \lor
  \ndia{n}\psi$;\quad$\nbox{n}\verum$;\label{glp:scheme:1}
\item $\ndia{n}\varphi \limpl \ndia{n}(\varphi \land
  \ndia{n}\lneg\varphi)$ (\emph{L\"ob's axiom});\label{glp:scheme:2}
\item $\ndia{m}\varphi \limpl \nbox{n}\ndia{m}\varphi$, for
  $m < n$;\label{glp:scheme:drop}
\item $\ndia{n}\varphi \limpl \ndia{m}\varphi$, for
  $m < n$ (\emph{monotonicity}); and \label{glp:scheme:3}
\item modus ponens and
  $\varphi \limpl \psi/\ndia{n}\varphi \limpl \ndia{n}\psi$.
\end{enumerate}

$\msGLP$ is formulated over a propositional language that contains
variables each being assigned a unique \emph{sort} $\alpha$, where
$0 \leq \alpha \leq \omega$. Let us formalize this notion more
carefully. Let $\mathbb{P}$ denote a fixed, countably infinite set of
\emph{propositional variables}. We fix a function
$\sort{\cdot} \colon \mathbb{P} \rightarrow \omega \cup \{\omega\}$ that
assigns a \emph{sort} $\alpha$ ($0 \leq \alpha \leq \omega$) to each
propositional variable $p \in \mathbb{P}$ in such a way that
$\mathbb{P}$ is partitioned into disjoint, countably infinite sets
$\mathbb{P}_0,\mathbb{P}_1,\ldots,\mathbb{P}_k,\ldots,\mathbb{P}_{\omega}$,
where
\begin{align*}
 p \in \mathbb{P}_\alpha\, \defequiv\, \sort{p} = \alpha, \quad
0 \leq \alpha \leq \omega.
\end{align*}
Formulas in this sorted language (i.e., formulas over variables
from $\mathbb{P}$) are called \emph{many-sorted formulas}. When it is
clear from context that we are dealing with many-sorted formulas, we
shall, however, often refer to them as ``formulas''.

The notion of sort is recursively extended to the set of all polymodal
formulas as follows:
\begin{itemize}
  \item $\verum$ and $\falsum$ have sort $0$;
  \item $\varphi \land \psi$ has
    sort $\max\{\alpha, \beta\}$ if $\varphi$ and $\psi$ have the
    respective sorts $\alpha$ and $\beta$;
  \item $\lneg\varphi$ has sort $1 + \alpha$ if $\varphi$
    has sort $\alpha$; and
  \item $\ndia{n}\varphi$ has sort $n$, for $n < \omega$ and any choice
    of $\varphi$.
\end{itemize}
It is easy to see that the sort of a formula is uniquely determined by
the sorts of its constituent propositional variables and we denote by
$\sort{\varphi}$ the sort of $\varphi$. The sort $\omega$ is included to
provide variables that can explicitly be assigned an arbitrary
arithmetical sentence in an arithmetical realization. In contrast,
variables of finite sort $n < \omega$ can be assigned arithmetical
$\Pi_{n+1}$-sentences only. Note that if $\sort{\varphi} = \omega$, then
also $\sort{\lneg\varphi} = \omega$. Moreover, notice that even formulas
equivalent in propositional logic may have different sorts, e.g., if $p$
has sort $n \in \omega$, then $\lneg\lneg p$ has sort $n + 2$.


\begin{definition}
The logic $\msGLP$ is axiomatized by the
schemas~\ref{glp:scheme:taut}, \ref{glp:scheme:1}, \ref{glp:scheme:2},
and~\ref{glp:scheme:3} of $\GLP$, as well as the following axiom schema:
\begin{enumerate}[resume]
\item $\ndia{n}\varphi \limpl \varphi$, if
  $\sort{\varphi} \leq n$
  (\emph{$\Sigma_{n+1}$-completeness}).\label{glp:scheme:4}
\end{enumerate}
Furthermore, $\msGLP$ is closed under modus ponens and
$\varphi \limpl \psi/\ndia{n}\varphi \limpl \ndia{n}\psi$, while
$\msGLP$ is not closed under arbitrary substitutions of formulas, but
one must rather respect the sorts of the propositional variables and
formulas involved. That is, one can only substitute formulas of sort at
most $\alpha$ for propositional variables of sort $\alpha$.
\end{definition}

Regarding the omission of axiom schema~\ref{glp:scheme:drop}, note that,
for $m < n$,
$\msGLP \proves \ndia{n}\lneg\ndia{m}\varphi \limpl
\lneg\ndia{m}\varphi$, whence
$\msGLP \proves \ndia{m}\varphi \limpl \nbox{n}\ndia{m}\varphi$ follows
by propositional logic. Hence, $\msGLP$ extends $\GLP$ in the sense
that, for any formula $\varphi$ in the language of $\GLP$, if
$\GLP \proves \varphi$, then $\msGLP \proves \varphi'$, where $\varphi'$
is obtained from $\varphi$ by arbitrarily assigning sorts to
propositional variables.

The logic $\GLP$ is not complete for any class of Kripke
frames~\citep{ignatiev:1993}. Therefore, the second author of this paper
considered in~\citep{beklemishev:2010} a weaker logic $\J$ that is
complete with respect to a natural class of Kripke frames and to which
$\GLP$ is reducible.\footnote{\citet{ignatiev:1993} also
  considered a weaker logic than $\GLP$ that is complete for a class of
  Kripke models and provided a reduction of $\GLP$ to that logic in
  order to establish arithmetical completeness. However, the
  arithmetical completeness proof in~\citep{beklemishev:2011}, where
  $\J$ is used, seems to be more convenient for our purposes.} We do so
as well and define a many-sorted counterpart $\msJ$ of $\J$ which arises
from $\msGLP$ by dropping the monotonicity axiom
schema~\ref{glp:scheme:3} and adding the schema
\begin{enumerate}[resume]
\item $\ndia{m}\ndia{n}\varphi \limpl \ndia{m}\varphi$, for
  $m < n$.\label{j:scheme:1}
\end{enumerate}
Using monotonicity (schema~\ref{glp:scheme:3}), we infer
$\msGLP \proves \ndia{m}\ndia{n}\varphi \limpl \ndia{m}\ndia{m}\varphi$,
whence by
$\msGLP \proves \ndia{m}\ndia{m} \varphi \limpl \ndia{m}\varphi$, we see
that schema~\ref{j:scheme:1} above is provable in $\msGLP$, i.e.,
$\msGLP$ extends $\msJ$. We remark that the definition of
$\J$ in~\citep{beklemishev:2010} also comprises the axiom schema

\begin{enumerate}[resume]
\item $\ndia{n}\ndia{m}\varphi \limpl \ndia{m}\varphi$, for $n > m$. \label{j:scheme:o}
\end{enumerate}
This schema is readily proved in $\msJ$ using one instance
of~\ref{glp:scheme:4}---notice that $\sort{\ndia{m}\varphi} = m$.

\begin{remark}
  We would like to emphasize that formulas in the language of $\msGLP$
  and $\msJ$ are formulated in a different language than in their
  respective one-sorted versions $\GLP$ and $\J$. Hence, formally, the
  many-sorted logics and their one-sorted counterparts talk about
  different objects. However, if we claim that a one-sorted logic proves
  a many-sorted formula, we mean that the one-sorted logic proves the
  formula which results from the many-sorted one if we simply disregard
  the sorts and treat it as a one-sorted formula in the usual sense.
\end{remark}

\subsection{Kripke Models}

A \emph{(Kripke) frame} is a structure
$\fk{F} = (W, \{R_n\}_{n\geq 0})$, where $W$ is a non-empty set of
\emph{worlds} and each $R_k$, for $k \geq 0$, is a binary relation on
$W$. The frame $\fk{F}$ is called \emph{finite} if $W$ is finite and
$R_k = \emptyset$ for all but finitely many $k \geq 0$.

A \emph{valuation $\val{\cdot}$ on} a frame $\fk{F}$ maps every
propositional variable $p$ to a subset $\val{p} \subseteq W$. A
\emph{(Kripke) model} $\fk{A} = (W, \{R_n\}_{n \geq 0}, \val{\cdot})$ is
a triple such that $\fk{F} \coloneqq (W,\{R_n\}_{n \geq 0})$ is a Kripke
frame and $\val{\cdot}$ a valuation on $\fk{F}$. We say that $\fk{A}$ is
\emph{based on} $\fk{F}$.

Given any Kripke model
$\fk{A} = (W, \{R_n\}_{n \geq 0}, \val{\cdot})$, we extend
the valuation $\val{\cdot}$ recursively to the class of all polymodal
formulas:
\begin{itemize}
  \item $\val{\verum} = W$;\quad$\val{\falsum} = \emptyset$;
  \item $\val{\psi \land \chi} = \val{\psi} \cap \val{\chi}$;
  \item $\val{\lneg\psi} = W \setminus \val{\psi}$, and
  \item $\val{\ndia{n}\psi} = \{x \in W \mid \exists y\, (x R_n y \mand y \in \val{\psi})\}$.
\end{itemize}
We often write $\fk{A}, x \models \varphi$ instead of
$x \in \val{\varphi}$. We say that $\varphi$ is \emph{valid in
  $\fk{A}$}, denoted by $\fk{A} \models \varphi$, if
$\fk{A}, x \models \varphi$, for every $x \in W$. Moreover, for a frame
$\fk{F}$, we say that $\varphi$ is \emph{valid in $\fk{F}$}, if
$\varphi$ is valid in every model based on $\fk{F}$.

A binary relation $R$ on $W$ is \emph{conversely well-founded} if there
is no infinite chain of elements of $W$ of the form $x_0Rx_1Rx_2\cdots$.
It is easy to see that, for finite $W$, this condition is equivalent to
$R$ being irreflexive. A Kripke frame
$\fk{A} = (W, \{R_n\}_{n\geq0})$ is called a
\emph{$\J$-frame}~\citep{beklemishev:2010} if
\begin{enumerate}[label={(\alph*)}]
  \item $R_k$ is conversely well-founded and transitive, for all $k \geq 0$;
  \item $\forall x, y\,(x R_n y \Rightarrow \forall z\,(x R_m z
    \Leftrightarrow y R_m z))$, for $m < n$; and\label{frame:2}
  \item $\forall x, y, z\,(x R_m y \mand y R_n z \Rightarrow x R_m z)$,
    for $m < n$.\label{frame:3}
\end{enumerate}
A \emph{$\J$-model} is a Kripke model that is based on a $\J$-frame. The
fact that the $R_k$ must be conversely well-founded and transitive is a
classical property required to validate all instances of L\"ob's axiom
(schema~\ref{glp:scheme:2}). Frame condition~\ref{frame:2}
corresponds to the schemas~\ref{j:scheme:o}
and~\ref{glp:scheme:drop}, while frame condition~\ref{frame:3}
corresponds to schema~\ref{j:scheme:1}.

\begin{theorem}[\citep{beklemishev:2010}]
  \label{thm:jcompl}
  For any polymodal formula $\varphi$, $\J \proves \varphi$ iff
  $\varphi$ is valid in all $\J$-frames.
\end{theorem}

We call a $\J$-model
$\fk{A} = (W, \{R_n\}_{n \geq 0}, \val{\cdot})$ a $\msJ$-model, if it is
\emph{strongly persistent}, that is, if it satisfies the following two conditions:
\begin{enumerate}[label={(\arabic*)}]
\item if $\sort{p} \leq n$ and $\fk{A}, y \models p$, then
  $\fk{A}, x \models p$ whenever $x R_n y$; and\label{pers:1}
\item if $\sort{p} < n$ and $\fk{A}, y \not\models p$, then
  $\fk{A}, x \not\models p$ whenever $x R_n y$. \label{pers:2}
\end{enumerate}
Note that, up to now, the notion of strong persistence is the first
semantic notion that refers to sorts of variables at all. Sorts thus
have no realization on the frame level, but are rather present through
the notion of strong persistence on the level of
models. Condition~\ref{pers:1} states that truth of propositional
variables of sort at most $n$ must be propagated downwards along
$R_n$-arcs. Likewise, condition~\ref{pers:2} states that falsehood of
propositional variables having sort (strictly) less than $n$ must be
propagated downwards along $R_n$-arcs.

Having both conditions in place allows us to extend
\ref{pers:1} and~\ref{pers:2} to all sorted formulas. This
relationship between strong persistence and satisfaction of sorted
formulas is the content of the following lemma.

\begin{lemma}\label{lemma:strongpersistence}
  Let $\fk{A} = (W, \{R_n\}_{n \geq 0}, \val{\cdot})$ be a
  $\J$-model. Then $\fk{A}$ is strongly persistent iff for all formulas
  $\varphi$ and all $n \geq 0$ we have
  \begin{itemize}
  \item if $\sort{\varphi} \leq n$, then $xR_n y$ and $\fk{A}, y \models
    {\varphi}$ imply $\fk{A}, x \models {\varphi}$; and
  \item if $\sort{\varphi} < n$, then $xR_n y$ and
    $\fk{A}, y \not\models {\varphi}$ imply
    $\fk{A}, x \not\models {\varphi}$.
  \end{itemize}
\end{lemma}
\begin{proof}
  The proof is by induction on the structure of $\varphi$. The base case
  follows immediately by the definition of strong persistence. Assume
  $\varphi = \lneg \psi$ for some $\psi$. Suppose first that
  $\sort{\varphi} \leq n$ and $x R_n y$ such that
  $\fk{A},y \models \lneg \psi$. It follows that
  $\fk{A},y \not\models \psi$, and by
  $\sort{\psi} < \sort{\varphi} \leq n$ and the induction hypothesis, we
  infer that $\fk{A}, x \not\models \psi$, whence
  $\fk{A}, x \models \lneg \psi$ follows as required. The case where
  $\sort{\varphi} < n$ is handled similarly.

  Suppose now that $\varphi = \ndia{k}\psi$, for some $k \in
  \omega$.
  Assume that $\sort{\ndia{k}\psi} \leq n$ and let $x, y \in W$ be such
  that $x R_n y$ and $\fk{A},y \models \ndia{k}\psi$. We know that
  $\sort{\ndia{k}\psi} = k$, whence $k \leq n$ follows. Let $z \in W$ be
  such that $y R_k z$ and $\fk{A},z \models \psi$. Now frame
  condition~\ref{frame:2} and the fact that $R_k$ is transitive (for the
  case $k = n$) give us $x R_k z$, whence
  $\fk{A}, x \models \ndia{k}\psi$ follows as desired. Suppose now that
  $\sort{\ndia{k}\psi} < n$ and let $x, y \in W$ be such that $x R_n y$
  and $\fk{A},y \not\models \ndia{k} \psi$. Suppose to the contrary that
  $\fk{A},x \models \ndia{k}\psi$. Let $z \in W$ be such that $x R_k z$
  and $\fk{A},z \models \psi$. We know that $k < n$ and by frame
  condition~\ref{frame:2} we infer that $y R_k z$. Therefore,
  $\fk{A}, y \models \ndia{k}\psi$, a contradiction. Hence,
  $\fk{A},x \not\models \ndia{k}\psi$ as required.

  Suppose that $\varphi = \psi \land \chi$ for some formulas
  $\chi,\psi$. If $\sort{\varphi} \leq n$, $x R_n y$, and
  $\fk{A},y \models \varphi$, then
  $\sort{\psi},\sort{\chi} \leq \sort{\varphi}$, whence by
  $\fk{A},y \models \psi$, $\fk{A},y\models \chi$ and the induction
  hypothesis it follows that $\fk{A},x \models \varphi$, as
  required. Suppose that $\sort{\varphi} < n$. Then also
  $\sort{\psi},\sort{\chi} < n$, and if $\fk{A}, y \not\models \varphi$,
  then $\fk{A},y \not\models \psi$ or $\fk{A},y \not\models \chi$. In
  both cases, the induction hypothesis yields
  $\fk{A},x \not\models \varphi$. This finishes the case of
  conjunction.
\end{proof}

Note that, in the proof above, it is of importance that $\fk{A}$ is
indeed a $\J$-model. In particular, we require that $\fk{A}$ satisfies
frame condition~\ref{frame:2} and the fact that all $R_n$ are
transitive.

\begin{lemma}
  \label{lemma:spvalid}
  The axiom schema $\ndia{n}\varphi \limpl \varphi$ is valid in a
  $\J$-model $\fk{A}$ for all $\varphi$ such that
  $\sort{\varphi} \leq n$ iff $\fk{A}$ is strongly persistent.
\end{lemma}
\begin{proof}
  Assuming $\fk{A},x \models \ndia{n} \varphi$ gives us
  $\fk{A},y \models \varphi$ for some $y$ such that $x R_n y$, whence
  $\fk{A}, x \models \varphi$ follows by $\sort{\varphi} \leq n$ and one
  application of~\Cref{lemma:strongpersistence}.

  Conversely, if $\fk{A}$ satisfies all instances of
  $\ndia{n}\varphi \limpl \varphi$ ($\sort{\varphi} \leq n$), it
  satisfied these instances for all appropriate propositional variables
  and their negations (respecting their sorts). Hence, if
  $\sort{p} \leq n$, $\fk{A},y \models p$, and $x R_n y$, then by
  $\fk{A}, x \models \ndia{n}p \limpl p$ also $\fk{A},x \models
  p$.
  Likewise, if $\sort{p} < n$, $\fk{A}, y \not\models p$, and $x R_n y$,
  then $\fk{A}, y \models \lneg p$, whence by
  $\fk{A},x \models \ndia{n} \lneg p \limpl \lneg p$ it follows that
  $\fk{A},x \not\models p$ as needed.
\end{proof}

Our goal is now to show that $\msJ$ is sound and complete for the class
of all strongly persistent $\msJ$-models. Soundness follows by a
straightforward induction on the length of a derivation
invoking~\Cref{lemma:spvalid}. For proving completeness, we aim at a
reduction of $\msJ$ to $\J$ as detailed in the following.\footnote{The
  authors are thankful to one of the anonymous referees who pointed out
  a simplification of the completeness proof for $\msJ$.}

\subsection{Completeness of $\msJ$}

Let $\varphi$ be a many-sorted formula and let $p_1,\ldots,p_k$ exhaust
all variables from $\varphi$ and let $\alpha_1,\ldots,\alpha_k$ be their
respective sorts. Furthermore, let $\Theta \subset \omega$ be a finite
set of natural numbers. Define
\begin{align*}
  P_\Theta(\varphi) \coloneqq \bigwedge_{\mathclap{i = 1}}^k \bigwedge
  (\{\ndia{j}p_i \limpl p_i \mid j \in \Theta, j \geq \alpha_i\} \cup
  \{\ndia{j}\lneg p_i \limpl \lneg p_i \mid j \in \Theta, j >
  \alpha_i\})
\end{align*}
and
\begin{align*}
  P^+_\Theta(\varphi) \coloneqq P_\Theta(\varphi) \land
  \bigwedge_{\mathclap{j \in \Theta}}\nbox{j}P_\Theta(\varphi).
\end{align*}
If $\Theta$ consists of exactly those $n$ such that $\ndia{n}$ occurs as
a modality in $\varphi$, then we omit the subscript ``$\Theta$'' in the
expression $P^+_\Theta(\varphi)$ and write $P^+(\varphi)$ instead. A
similar convention is applied to $P_\Theta(\varphi)$. Intuitively, the
formula $P^+(\varphi)$ should ensure, when valid in a model, that the
model at hand is strongly persistent:

\begin{lemma}
  \label{lem:persistenceR}
  Suppose $\fk{A} = (W, \{R_n\}_{n \geq 0}, \val{\cdot})$ is a finite
  model such that $\fk{A} \models P_\Theta(\varphi)$, where $\Theta$
  is chosen such that $R_n \neq \emptyset$ implies $n \in \Theta$. Then
  $\fk{A}$ is strongly persistent.
\end{lemma}
\begin{proof}
  Let $\fk{A} = (W, \{R_n\}_{n \geq 0}, \val{\cdot})$ be a model and
  suppose $\fk{A} \models P_\Theta(\varphi)$. Consider any variable
  $p$ such that $\sort{p} \leq n$ and some $x, y \in W$ such that
  $\fk{A},y \models p$ and $x R_n y$. By the construction of
  $P_\Theta(\varphi)$, we know that
  $\fk{A}, x \models \ndia{n}p \limpl p$ and so $\fk{A},x \models p$ as
  required. Likewise, if $\sort{p} < n$, $\fk{A}, y \not\models p$, and
  $x R_n y$, then $P_\Theta(\varphi)$ contains the conjunct
  $\ndia{n} \lneg p \limpl \lneg p$, whence $\fk{A},x \models \lneg p$
  and thus $\fk{A}, x \not\models p$ follows.
\end{proof}

Note that the finiteness of the model in~\Cref{lem:persistenceR} is
essential, since otherwise $P^+_\Theta(\varphi)$ may not be finite.

Let $\fk{A} = (W, \{R_n\}_{n \geq 0},\val{\cdot})$ be a $\J$-model. A
\emph{root} of $\fk{A}$ is a world $r \in W$ such that for all
$x \in W$, there is a $k \geq 0$ such that $r R_k x$ or $r = x$. A model
which has a root is called \emph{rooted}.

\begin{lemma}[\citep{beklemishev:2011}]
  \label{lem:rootedjmodel}
  If $\J \nproves \varphi$, then there is a finite $\J$-model
  $\fk{A}$ with root $r$ such that $\fk{A},r \not\models
  \varphi$.
  Moreover, one can choose $\fk{A}$ such that $R_n \neq \emptyset$
  implies that $\ndia{n}$ occurs in $\varphi$.
\end{lemma}

\begin{corollary}
  \label{lem:treemodel}
  If $\msJ \nproves \varphi$, then there is a finite $\msJ$-model
  $\fk{A}$ with root $r$ such that $\fk{A},r \not\models \varphi$.
\end{corollary}
\begin{proof}
  Suppose $\msJ \nproves \varphi$. Then also
  $\J \nproves P^+(\varphi) \limpl \varphi$, since
  $\msJ \proves P^+(\varphi)$ and $\msJ$ extends $\J$. Using
  again~\Cref{lem:rootedjmodel}, we know that there is a $\J$-model
  $\fk{A} = (W, \{R_n\}_{n \geq 0}, \val{\cdot})$ with root $r$ such
  that $\fk{A},r \not\models P^+(\varphi) \limpl \varphi$. Furthermore,
  $R_n \neq \emptyset$ implies that $\ndia{n}$ occurs in
  $P^+(\varphi) \limpl \varphi$. Hence, $R_n \neq \emptyset$ also
  implies that $\ndia{n}$ occurs as modality in $\varphi$, since
  $P^+(\varphi)$ and $\varphi$ contain the same modalities. Since
  $\fk{A},r \models P^+(\varphi)$ and $r$ is the root of $\fk{A}$, we
  can infer that $\fk{A},x \models P(\varphi)$, for all $x \in
  W$.
  By~\Cref{lem:persistenceR} it follows that $\fk{A}$ is strongly
  persistent, i.e., $\fk{A}$ is a $\msJ$-model having root $r$ such that
  $\fk{A}, r \not\models \varphi$. This proves the claim.
\end{proof}

From this, the completeness of $\msJ$ for the class of $\msJ$-models
follows immediately:

\begin{corollary}
  $\msJ \proves \varphi$ iff $\varphi$ is valid in all $\msJ$-models.
\end{corollary}
\begin{proof}
  Soundness is an easy induction on the length of a
  derivation. Completeness follows immediately by~\Cref{lem:treemodel}.
\end{proof}

\subsection{Formal Arithmetic}

We consider first-order theories in the language of arithmetic. The
theories we consider are extensions of Peano arithmetic ($\PA$). The
class of \emph{$\Delta_0$-formulas} are all formulas where each
occurrence of a quantifier is of one of the forms
\begin{align*}
  \forall x \leq t\, \varphi &\coloneqq \forall x\,(x \leq t \limpl
  \varphi) \text{ or }\\
  \exists x \leq t\, \varphi &\coloneqq \exists x\,(x \leq t \land \varphi),
\end{align*}
where $t$ is a term that has no occurrence of the variable
$x$. Occurrences of such quantifiers are called \emph{bounded}, and we
often call $\Delta_0$-formulas simply \emph{bounded formulas}. The
classes of $\Sigma_n$- and \emph{$\Pi_n$-formulas} are defined
inductively as follows: $\Sigma_0$- and $\Pi_0$-formulas are the same as
$\Delta_0$-formulas. If $\varphi(\vec{x}, y)$ is a $\Pi_n$-formula, then
$\exists y\,\varphi(\vec{x}, y)$ is a
$\Sigma_{n+1}$-formula. Accordingly, if $\varphi(\vec{x}, y)$ is a
$\Sigma_n$-formula, then $\forall y\,\varphi(\vec{x}, y)$ is a
$\Pi_{n+1}$-formula. A formula is in $\Delta_{n+1}$ iff it is both in
$\Sigma_{n+1}$ and $\Pi_{n+1}$. When an arithmetical theory $T$ is
given, we often identify these classes modulo provable equivalence in
$T$. In this context, we say that a formula is \emph{$\Sigma_{n}$ in
  $T$} ($\Pi_n$, $\Delta_n$, respectively), if it is provably equivalent
to a $\Sigma_n$-formula ($\Pi_n$-formula, $\Delta_n$-formula,
respectively) in $T$.

We denote by $\ol{n}$ the \emph{$n$-th numeral} that represents the
number $n$ in our arithmetical language (when reasoning in an
arithmetical theory, we shall often write simply $n$ instead of
$\ol{n}$). We assume a standard global assignment $\code{\cdot}$ of
expressions (terms, formulas, etc.)  to natural numbers, called the
\emph{codes} of the respective expressions. When presenting formulas in
the arithmetical language, we usually write $\code{\tau}$ instead of
$\ol{\code{\tau}}$. We often consider primitive recursive families of
formulas $\varphi_n$ that depend on a parameter $n \in \omega$. In this
context, $\code{\varphi_x}$ denotes a primitive recursive definable term
with free variable $x$ whose value for a given $n$ is the G\"odel number
of $\varphi_n$. In particular, the expression $\code{\varphi(\dot{x})}$
denotes a primitive recursive definable term whose value given any $n$
is the G\"odel number of $\varphi(\ol{n})$, i.e., the G\"odel number of
the formula resulting from $\varphi$ when substituting the term $\ol{n}$
for $x$.


A theory $T$ is \emph{sound} if $T \proves \varphi$ implies
$\stdmodel \models \varphi$, for every arithmetical sentence
$\varphi$. A theory $T$ is \emph{axiomatizable} if $T$ has a
recursive set of axioms. For an axiomatizable extension $T$ of $\PA$, we
denote by $\Box_T(\alpha)$ the formula that formalizes the notion of
provability in $T$ in the usual sense.\footnote{We assume that Greek
  letters $\alpha,\beta,\ldots$ range over codes of arithmetical
  formulas.} We write $\Box_T\varphi$ instead of
$\Box_T(\code{\varphi})$. The formula $\Box_T$ defines the standard
G\"odelian provability predicate for $T$. More generally, given a
formula $\Prov(\alpha)$ with one free variable $\alpha$, we say that
$\Prov$ is a \emph{provability predicate of level $n$ over
  $T$}~\citep{ignatiev:1993}, if for all arithmetical sentences
$\varphi, \psi$:
\begin{enumerate}[label={(\alph*)}]
  \item $\Prov$ is a $\Sigma_{n+1}$-formula;
  \item $T \proves \varphi$ implies $\PA \proves \Prov(\code{\varphi})$;
  \item $\PA \proves \Prov(\code{\varphi \limpl \psi}) \limpl
    (\Prov(\code{\varphi}) \limpl \Prov(\code{\psi}))$; and
  \item if $\varphi$ is a $\Sigma_{n+1}$-sentence, then
    $\PA \proves \varphi \limpl \Prov(\code{\varphi})$ (\emph{provable
      $\Sigma_{n+1}$-completeness}).\label{sigma1compl}
\end{enumerate}

It is well-known that $\Box_T$, in its standard formulation, is a
provability predicate of level $0$.  A provability predicate $\Prov$ is
\emph{sound} if $\stdmodel \models \Prov(\code{\varphi})$ implies
$\stdmodel \models \varphi$, for every arithmetical sentence
$\varphi$. A sequence $\pi$ of formulas $\Prov_0, \Prov_1,\ldots$ is a
\emph{strong sequence of provability predicates over $T$}, if there is a
sequence $r_0 < r_1 < \cdots$ of natural numbers  such that, for all
$n \geq 0$,
\begin{itemize}
  \item $\Prov_n$ is a provability predicate of level $r_n$ over $T$; and
  \item
    $T \proves \Prov_n(\code{\varphi}) \limpl
    \Prov_{n+1}(\code{\varphi})$,
    for any arithmetical sentence $\varphi$.
\end{itemize}
We write $\nbox{n}_\pi\varphi$ for $\Prov_n(\code{\varphi})$. Moreover,
the dual of $\nbox{n}_\pi$ is defined by
$\ndia{n}_\pi\varphi \coloneqq \lneg\nbox{n}_\pi\lneg\varphi$. Given
such a sequence $\pi$, we denote by $\sort{\pi_n}$ the level of the
$n$-th provability predicate of $\pi$.

Since the provability predicate $\nbox{n}_\pi$ from $\pi$ is a
$\Sigma_k$-sentence for some $k > 0$, we can associate (in analogy to
the standard G\"odelian provability predicate) a predicate
$\Prf^\pi_n(\alpha, y)$ which expresses the statement ``\emph{$y$ codes a proof of
$\alpha$}'' and
\begin{align*}
  T \proves \Prov_n(\alpha) \lequiv \exists y \, \Prf^\pi_n(\alpha, y).
\end{align*}
We assume that $\Prf_n^\pi$ is chosen in such a way such that every number
$y$ codes a proof of at most one formula and that every provable formula
has arbitrarily long proofs.

We denote by $\True_{\Pi_n}(\alpha)$ the well-known truth-definition for
the class of all $\Pi_{n}$-sentences, i.e., $\True_{\Pi_n}(\alpha)$
expresses the statement ``\emph{$\alpha$ is the G\"odel number of a true
  arithmetical $\Pi_n$-sentence}''. The truth-definition for
$\Pi_n$-sentences serves as a basis for a natural strong sequence of
provability predicates. Let $\nbox{0}_T \coloneqq \Box_T$ and
\begin{align*}
  \nbox{n+1}_T(\alpha) \coloneqq \exists \beta\,(\True_{\Pi_n}(\beta)
  \land \Box_T(\beta \limpl \alpha)),\quad\text{for $n \geq 0$}.
\end{align*}
The formula $\nbox{n}_T$ is a provability predicate of level $n$. It
formalizes the notion of being provable in the theory
$T + \Th_{\Pi_n}(\stdmodel)$, where $\Th_{\Pi_n}(\stdmodel)$ is the set
of all true $\Pi_n$-sentences.

Another strong sequence of provability predicates is defined by
$\nbox{0}_\omega \coloneqq \Box_\PA$ and
\begin{align*}
  \nbox{n+1}_\omega \coloneqq \exists \beta\,(\forall x\,\nbox{n}_\omega
  \beta(\dot{x}) \land \nbox{n}_\omega(\forall x\,\beta(x) \limpl
  \alpha)),\quad\text{for $n \geq 0$}.
\end{align*}
The predicate $\nbox{n}_\omega$ is of level $2n$ and formalizes the
notion of ``\emph{provability by $n$ applications of the
  $\omega$-rule}''. Japaridze originally showed arithmetical
completeness of $\GLP$ for this interpretation, while completeness with
respect to the broader class of interpretations, defined by strong
sequences of provability predicates, was later established
in~\citep{ignatiev:1993}.\footnote{See~\citep{artemovbekl:2004} for a
  brief historical background.}

\subsubsection{Arithmetical Interpretation}

 An (\emph{arithmetical})
\emph{realization} is a function $f$ that maps propositional variables
to arithmetical sentences. Let $\pi$ be a strong sequence of provability
predicates over $T$. The realization $f$ is \emph{typed for $\pi$}, if,
for every propositional variable $p$, $f(p)$ is an arithmetical
$\Pi_{\sort{\pi_n} + 1}$-sentence, provided $n = \sort{p} < \omega$. (We
shall simply say that $f$ is \emph{typed} if $\pi$ is clear from
context.) Any realization $f$ can be uniquely extended to a map $f_\pi$
that is defined for all polymodal formulas as follows:
\begin{itemize}
  \item $f_\pi(\falsum) = \falsum$;\quad$f_\pi(\verum) =
    \verum$, where $\falsum$ (resp.,~$\verum$) is a convenient
    contradictory (resp.,~tautological) statement in the language of arithmetic;
  \item $f_\pi(p) = f(p)$, for any propositional variable $p$;
  \item $f_\pi(\cdot)$ commutes with the propositional
    connectives; and
  \item
    $f_\pi(\ndia{n}\varphi) = \ndia{n}_\pi {f}_\pi(\varphi)$,
    for all $n \geq 0$.
\end{itemize}
(Notice that we include the subscript $\pi$ in $f_\pi$ since $f_\pi$
depends on the choice of $\pi$ due to the fourth item above.) By some
simple closure properties of the class of $\Pi_n$-sentences, it follows
that $\sort{\varphi} = n$ implies that $f_\pi(\varphi)$ is provably
equivalent to a $\Pi_{\sort{\pi_n} + 1}$-sentence in $T$. 

$\msGLP$ is arithmetically sound for this semantics:

\begin{lemma}
  \label{lemma:glparithsound}
  Let $T$ be an axiomatizable extension of $\PA$. If
  $\msGLP \proves \varphi$, then $T \proves f_\pi(\varphi)$ for all
  arithmetical realizations $f$ that are typed for $\pi$.
\end{lemma}
\begin{proof}
  The lemma is shown by induction on the length of a proof of $\varphi$
  in $\msGLP$. Most of the axioms are clear. In particular, the
  provability of the instances of L\"ob's axiom (axiom
  schema~\ref{glp:scheme:2}) is well-known, although not trivial at all;
  see, e.g.,~\citep{boolos:1995,beklemishev:2004} for a proof. For the
  schema of $\Sigma_{n+1}$-completeness (schema~\ref{glp:scheme:4}),
  notice that $\ndia{n}\varphi \limpl \varphi$ is equivalent to
  $\lneg\varphi \limpl \nbox{n}\lneg \varphi$ in $\msGLP$. The sentence
  $\lneg f_\pi(\varphi)$ is now provably equivalent in $T$ to a
  $\Sigma_{|\pi_{n}|+1}$-sentence, whence
  $T \proves \lneg f_\pi(\varphi) \limpl \nbox{n}_\pi\lneg
  f_\pi(\varphi)$ and thus
  $T \proves \ndia{n}_\pi f_\pi(\varphi) \limpl f_\pi(\varphi)$ follows
  by provable $\Sigma_{|\pi_n| + 1}$-completeness
  (property~\ref{sigma1compl} above). The induction step, i.e., closure
  under the rules of inference, is easy to establish. We leave the
  details to the reader.
\end{proof}

Arithmetical completeness holds under the additional assumption of
soundness of the provability predicates involved. As already mentioned,
arithmetical completeness for $\GLP$ has first been established
in~\citep{japaridze:1988} and has been significantly extended and
simplified in~\citep{ignatiev:1993}. In fact, arithmetical
interpretations for variants of $\GLP$ have been pushed even further:
in~\citep{joostenf:2013}, a transfinite version $\GLP_\Lambda$ (for
$\Lambda$ a recursive ordinal) of $\GLP$ is considered, where one has a
modal operator $\nbox{\xi}$ for each $\xi < \Lambda$. The authors
of~\citep{joostenf:2013} show that $\GLP_\Lambda$ is sound and complete
for some suitable theories of second-order arithmetic
(see~\citep{joostenf:2013} for details), where $\nbox{\xi}\varphi$ is
interpreted as ``$\varphi$ is provable using an $\omega$-rule of depth
$\xi$''.

Regarding our intended arithmetical interpretation of $\GLP$,
in~\citep{beklemishev:2011}, the second author of this paper provided
yet another simplified proof for the arithmetical completeness theorem
for $\GLP$ that is close to Solovay's original construction for the
logic $\GL$~\citep{solovay:1976}. We are going to work along the lines
of the construction presented in~\citep{beklemishev:2011}, since it
seems to be the most convenient for our purpose. This is due to the fact
that, essentially, when redoing the construction for $\GLP$ carried out
in~\cite{beklemishev:2011} in the setting of $\msGLP$, we can observe
that the arithmetical realization $f$ one extracts from the fact that
$\msGLP \nproves \varphi$ is actually typed (for a previously chosen
strong sequence of provability predicates). Thus, in the next section,
we are first going to present the essentials of the arithmetical
completeness proof provided in~\citep{beklemishev:2011} and afterwards
observe that we can restrict ourselves to typed arithmetical
realizations.

\section{Arithmetical Completeness}
\label{sec:3}

Arithmetical completeness proofs usually rely on reasonable Kripke
semantics, since those proofs usually establish the following fact: if
$\varphi$ is a formula that has a Kripke model falsifying $\varphi$ in a
certain world, one can find an arithmetical realization such that the
arithmetical theory under consideration does not prove $\varphi$ under
this realization. Since $\GLP$ is, however, not complete for any class
of Kripke frames, in~\citep{beklemishev:2011}, $\GLP$ is reduced to $\J$
and then one relies on the Kripke semantics of $\J$ in order to prove
arithmetical completeness. Our strategy towards obtaining an arithmetical
completeness for $\msGLP$ is now as follows:
\begin{itemize}
\item We revisit the construction of~\citep{beklemishev:2011} and
  present the essentials contained in there. We will review all of the
  necessary information from~\citep{beklemishev:2011} needed to follow
  the new parts of the proof. For thorough details, we refer the
  interested reader to~\citep{beklemishev:2011}.
\item We observe that when this construction is carried out using
  $\msJ$-models rather than $\J$-models, we can extract an
  arithmetical realization that is actually typed (for a previously
  chosen strong sequence of provability predicates).
\end{itemize}

\subsection{Preliminary Preparations}
Before presenting the essentials of the construction
in~\citep{beklemishev:2011}, we shall introduce some additional notions.

Let $\varphi$ be a polymodal
formula. Following~\citep{beklemishev:2011}, we define auxiliary
formulas $M(\varphi)$ and $M^+(\varphi)$ as follows. Consider an
enumeration
$\ndia{m_1}\varphi_1,\ndia{m_2}\varphi_2,\ldots,\ndia{m_s}\varphi_s$ of
all subformulas of $\varphi$ of the form $\ndia{k}\psi$ and let
$n \coloneqq \max_{i \leq s} m_i$. Define
\begin{align*}
  M(\varphi) &\coloneqq \bigwedge_{\mathclap{\substack{1 \leq i \leq s \\ m_i < j
    \leq n}}}(\ndia{j}\varphi_i \limpl \ndia{m_i}\varphi_i),
\end{align*}
and, furthermore,
\begin{align*}
  M^+(\varphi) &\coloneqq M(\varphi) \land \bigwedge_{\mathclap{i \leq
      n}} \nbox{i} M(\varphi).
\end{align*}
Notice that $\msGLP \proves M^+(\varphi)$ by the use of the monotonicity
axiom schema~\ref{glp:scheme:3}.

The arithmetical completeness theorem we are going to establish
reads as follows:
\begin{theorem}
\label{th:msglpcompl}
  Let $T$ be an axiomatizable extension of $\PA$ and $\pi$ a strong
  sequence of provability predicates over $T$ whose predicates are all
  sound. Then, for all formulas $\varphi$, the following statements are
  equivalent:
  \begin{enumerate}[label={\rm (\arabic*)}]
    \item $\msGLP \proves \varphi$;\label{msglp:1}
    \item $\msJ \proves M^+(\varphi) \limpl \varphi$;\label{msglp:2}
    \item $T \proves f_\pi(\varphi)$, for all arithmetical realizations
      $f$ that are typed for $\pi$.\label{msglp:3}
  \end{enumerate}
\end{theorem}


It is clear that~\cref{msglp:2} implies~\ref{msglp:1} since
$\msGLP \proves M^+(\varphi)$ and $\msGLP$ extends $\msJ$. Moreover, we
have already established that~\ref{msglp:1} implies~\ref{msglp:3}
in~\Cref{lemma:glparithsound}. It thus remains to show
that~\ref{msglp:3} implies~\ref{msglp:2}.

Throughout the proof presented in this section, we fix an axiomatizable
extension $T$ of $\PA$ and a strong sequence of provability predicates
$\pi$ of which every provability predicate is sound. For a proof of the
arithmetical completeness theorem for $\msGLP$, we are going to argue by
contraposition and show that $\msJ \nproves M^+(\varphi) \limpl \varphi$
entails that there is a typed realization $f$ for $\pi$ such that
$T \nproves f_\pi(\varphi)$.

\subsection{Essentials of the Construction for $\GLP$} 

We fix a polymodal
formula $\varphi$ and assume that
$\msJ \nproves M^+(\varphi) \limpl \varphi$. Our goal here is to present
the essentials of the construction in~\citep{beklemishev:2011} in order
to obtain a realization $f$ such that $T \nproves
f_\pi(\varphi)$. Afterwards, we are going to show that $f$ is actually typed for
$\pi$. 

By~\Cref{lem:treemodel}, we know that there is a finite $\msJ$-model
$\fk{A} = (W, \{R'_n\}_{n \geq 0}, \val{\cdot})$ with root $r$ such that
$\fk{A}, r \models M^+(\varphi)$ and $\fk{A},r \not\models \varphi$. For
technical clarity, assume that $W = \{1,2,\ldots,N\}$ for some
$N \geq 1$ and $r = 1$. Construct a new model
$\fk{A}_0 = (W_0, \{R_n\}_{n \geq 0},\val{\cdot})$, where
\begin{itemize}
  \item $W_0 = \{0\} \cup W$;
    \item $R_0 = \{(0, x) \mid x \in W\} \cup R'_0$;
    \item $R_k = R'_k$, for $k > 0$; and
    \item $\fk{A}_0, 0 \models p\ \defequiv\ \fk{A}, 1 \models p$, for all
      variables $p$.
\end{itemize}
Notice that $\fk{A}_0$ is still a finite $\msJ$-model such that
$\fk{A}_0, r \not\models M^+(\varphi) \limpl \varphi$ ($r$ is, however,
not the root of $\fk{A}_0$ anymore). In particular, $\fk{A}_0$ is still
strongly persistent. Throughout the proof, let $m$ be the only number such
that $R_m \neq \emptyset$ and $R_k = \emptyset$, for all $k > m$.

 As in
\citep{beklemishev:2011}, we define the following auxiliary notions:
\begin{align*}
  R_k(x)  &\coloneqq \{y \mid x R_k y\}\rlap{,}\\
  R_k^\ast(x)  &\coloneqq \{y \mid y \in R_i(x), \text{ for some $i
                 \geq k$}\}\rlap{,} ~\mbox{~and}\\
  \widetilde{R}_k(x)  &\coloneqq R_k^\ast(x) \cup \bigcup\{R_k^\ast(z) \mid x
                  \in R^\ast_{k + 1}(z), z \in W_0\}\rlap{.}
\end{align*}
Note that $R_k(x) \subseteq R_k^\ast(x) \subseteq \widetilde{R}_k(x)$.
The set $\widetilde{R}_k(x)$ consists of 
\begin{enumerate*}[label={(\arabic*)}]
\item all $y$ that are $R_k^\ast$-reachable from $x$, and 
\item  all $y$ that are $R_k^\ast$-reachable from
some $z \in W_0$ such that $x$ is $R_{k+1}^\ast$-reachable from $z$.
\end{enumerate*}


The proof now proceeds by defining, for each $x \in W_0$, an
arithmetical sentence $S_x$ which expresses that a certain function
reaches a limit. More formally, suppose
$g \colon \omega \rightarrow W_0$ is a function that is coded by an
arithmetical formula $G(x, y)$ in $T$. We write $\ell^G = x$ for the
formula $\exists N_0 \forall n \geq N_0\; G(n, x)$, i.e., the formula
that expresses that $g$ reaches a \emph{limit} at point $x$. The proof
in~\citep{beklemishev:2011} relies on the construction of a sequence
$h_1,h_2,\ldots,h_m$ of functions that provably satisfy certain
properties stated in the lemma below.

 Before proceeding with the
statement of that lemma, let us clarify some notation first. Given an
arithmetical formula $\psi(x)$ and some $A \subseteq W_0$, we use
quantifier expressions of the form $\exists x \in A\ \psi(x)$,
$\forall x \in A\ \psi(x)$, etc., to respectively abbreviate finite
disjunctions $\bigvee_{x \in A} \psi(\ol{x})$ and finite conjunctions
$\bigwedge_{x \in A} \psi(\ol{x})$ over the elements of $A$; similar
conventions are employed for $\exists!x \in A\ \psi(x)$ (``\emph{there exists
exactly one $x \in A$ such that $\psi(x)$}''). When we know that
$F(\vec{x},y)$ defines a provably total function in $T$, we shall
furthermore often use expressions like $f(\vec{x}) \in A$ to abbreviate
a formula of the form $\bigvee_{y \in A} f(\vec{x}) = \ol{y}$, where $f$
is an abbreviation in the metalanguage for the function defined by
$F(\vec{x},y)$.

\begin{lemma}[\citep{beklemishev:2011}]
  \label{lem:hprops}
  There is a sequence of functions
  $h_0,h_1,\ldots,h_m \colon \omega \rightarrow W_0$ respectively
  defined by formulas $H_0,H_1,\ldots,H_m$ in $T$, i.e.,
  \begin{enumerate}[label={\rm (\arabic*)}]
  \item $T \proves \forall x\, \exists! w \in W_0\;
    H_k(x,w)$, \label{hprop:1}
  \item $T \proves \forall x,y\,(h_k(x) = y \lequiv H_k(x,y))$,
  \end{enumerate}
  such that the functions $h_0,h_1,\ldots,h_k$, provably
  in $T$, satisfy the following properties:
  \begin{align*}
    h_0(0) &= 0\text{ and } h_{k}(0) = \ell^{H_{k -1}} \text{, for $k = 1,\ldots,m$;}\\[9pt]
    h_k(n + 1) &= \left\{
        \begin{array}{l l}
          z, & \quad\text{if $h_k(n) R_k z$ and $\Prf_k(\code{\lneg
              S_z}, n)$,}\\[3pt]
          h_k(n), & \quad\text{otherwise.}
        \end{array} \right.
  \end{align*}
  Moreover, for $k = 0,1,\ldots,m$, $H_k$ is in $\Sigma_{|\pi_k| + 1}$ and
  the following properties hold:
  \begin{enumerate}[label={\rm (\arabic*)},resume]
    \item
      $T \proves \forall i, j\,\forall z \in W_0\, (i < j \land h_k(i) =
      z \limpl h_k(j) \in R_k(z) \cup \{z\})$,\label{hprop:2}
    \item $T \proves \exists!x \in W_0\ \ell^{H_k} = x$,\label{hprop:3}
    \item
      $T \proves \forall z \in W_0\,(\exists n\, h_k(n) = z \limpl
      \ell^{H_m} \in R^\ast_k(z) \cup \{z\})$.\label{hprop:4}
  \end{enumerate}
\end{lemma}
In the following, we fix a sequence of functions $h_0,h_1,\ldots,h_m$
respectively defined by formulas $H_0,H_1,\ldots,H_m$ with the
properties as stated in~\Cref{lem:hprops}. We let $S_x$ be an
abbreviation for $\ell^{H_m} = \ol{x}$.

Notice the self-referential character of the definition of the $h_k$ due
to their reference to the sentences $S_z$. \Cref{hprop:2}
of~\Cref{lem:hprops} above states that $h_k$ is weakly increasing along
$R_k$ (i.e., $h_k(n + 1)$ either has the value $h_k(n)$ or increases
with respect to $R_k$),~\cref{hprop:3} states that $h_k$ reaches a
unique limit, while~\cref{hprop:4} means that, knowing that $h_k(n) = z$
for some value $n$, we can conclude that the last function $h_m$ reaches
its limit either at $z$ or at some $x \in R_k^\ast(z)$ (this becomes
intuitively clear if we consider the fact that $h_{k+1}$ starts where
$h_k$ reaches its limit).

We give an intuitive explanation for the concepts introduced so far
using a metaphor.\footnote{The metaphor using travelers that follows is
  inspired by a similar one for the G\"odel-L\"ob logic $\GL$; see,
  e.g.,~\cite{artemovbekl:2004}.} Think of the domains of the functions
$h_k$ as being points in time, expressed via natural numbers. Moreover,
imagine that we have $m$ travelers who travel around in our model such that
the fact $h_k(n) = x$ expresses that traveler $k$ is at world
$x \in W_0$ at time instant $n$. The limit $\ell^{H_k}$ of $h_k$ can be
seen as a world where the $k$-th traveler stays indefinitely. Using this
metaphor, $h_k$ satisfies the following properties (justified
by~\Cref{lem:hprops}):
\begin{itemize}
\item Traveler $0$ starts at world $0$. Moreover, traveler $k + 1$
  starts where the $k$-th traveler stays indefinitely.
\item Traveler $k$ can only travel at time instant $n + 1$ to the world
  $z$ such that $h_k(n)R_k z$, if $n$ codes a proof that the \emph{last}
  traveler (i.e., traveler $m$) does not stay at world $z$
  indefinitely. Otherwise, she must stay at world $h_k(n)$.
\end{itemize}
Now if we consider the implicit constraints that our model under
consideration is finite and that the travelers cannot travel backwards
in our model, we would expect that honest travelers all stay at home
(i.e., at world $0$)---formally, we in particular expect that $S_0$ is
true in the standard model.

Having the notions from~\Cref{lem:hprops} in place, the use of the
relation $\widetilde{R}_k$ can be explained as follows. Assume (in $T$)
that $\ell^{H_m} = x$, where $x \neq 0$. That is, the last traveler $m$
stays in world $x$ indefinitely and $x$ is different from $0$. What can
we say about the set of worlds at which the last traveler can
$\nbox{k}_\pi$-provably stay indefinitely? Since $\ell^{H_m} \neq 0$,
one traveler has traveled at least one time from one world to
another. We certainly know that all the travelers $k, k+1,\ldots,m$
respectively travel along the relations
$R_k,R_{k+1},\ldots,R_m$. Furthermore, the $(n+1)$-st traveler starts
where the $n$-th stays indefinitely. Assuming that the $k$-th traveler
stays at $z$ indefinitely (i.e., $\ell^{H_k} = z$), we must thus have
that $\ell^{H_m} \in R^\ast_{k+1}(z) \cup \{z\}$, i.e., the last
traveler stays at some element from $R^\ast_{k+1}(z) \cup \{z\}$. This
then entails $R_k^\ast(z) \subseteq \widetilde{R}_k(x)$. Moreover,
$\ell^{H_k} = z$ implies $h(n) = z$ for some $n \geq 0$, whence
$\ell^{H_m} \in R_k^\ast(z) \cup \{z\}$ (\cref{hprop:4}
of~\Cref{lem:hprops}). But we know that $\ell^{H_m} \neq z$, since
otherwise $h_k$ could not attain the value $z$ (which is not equal to
$0$ since $x \neq 0$) at all. Therefore,
$\ell^{H_m} \in R^\ast_k(z) \subseteq \widetilde{R}_k(x)$.
Essentially, a formalization of this argument constitutes a proof
of~\cref{sprop:3} of~\Cref{lem:sprops} below. Thus, $\widetilde{R}_k(x)$
consists of all the worlds which could be ($\nbox{k}_\pi$-provably)
possible candidates for the last traveler to stay indefinitely, provided
we assume $\ell^{H_m} = x$ (i.e., $S_x$) for $x \neq 0$ in $T$.

The following lemma makes the notions discussed informally using our
metaphor more precise.

\begin{lemma}[\citep{beklemishev:2011}]
  \label{lem:sprops}
  The sentences $S_x$ satisfy the following properties:
  \begin{enumerate}[label={\rm (S\arabic*)}]
  \item $T \proves \bigvee_{x \in W_0} S_x$ and
    $T \proves \lneg(S_x \land S_y)$, for $x \neq y$;\label{sprop:1}
    \item $T \proves S_x \limpl \ndia{k}_\pi S_y$, for all $y$ such that
      $xR_ky$;\label{sprop:2}
    \item $T \proves S_x \limpl \nbox{k}_\pi(\bigvee_{y \in \widetilde{R}_k(x)} S_y)$,
      for all $x \neq 0$; and\label{sprop:3}
    \item $\stdmodel \models S_0$.\label{sprop:4}
  \end{enumerate}
\end{lemma}
\begin{proof}
  For the sake of clarity, let us repeat some parts of the proof
  from~\citep{beklemishev:2011}. \Cref{sprop:1} states that $h_m$
  reaches its limit at one and only one world in the model
  $\fk{A}_0$. Notice that~\ref{sprop:1} follows immediately
  by~\cref{hprop:1} of ~\Cref{lem:hprops}.

  \Cref{sprop:2} expresses the
  fact that, assuming $S_x$ in $T$, for all $y$ such that $x R_k y$, one
  can consistently assume (regarding the $k$-th provability predicate of
  $\pi$) that $h_m$ converges to $y$. One can prove this item by
  formalizing the following argument in $T$:
\begin{quote}
  Assume $S_x$ and $\nbox{k}_\pi \lneg S_y$ for some $y$ such that
  $x R_k y$. Then either $\ell^{H_k} = x$ or $\ell^{H_k} = z$, for some
  $z \in R^\ast_{k+1}(x)$. In both cases, since $\fk{A}_0$ is a
  $\msJ$-model, we have that $R_k(x) = R_k(\ell^{H_k})$. Pick a number
  $n_0$ such that $\forall n \geq n_0\ h_k(n) = \ell^{H_k}$. Since
  $\nbox{k}_\pi \lneg S_y$, there is an $n_1 \geq n_0$ such that
  $\Prf^\pi_k(\code{\lneg S_y}, n_1)$. But $\ell^{H_k} R_k y$ and
  $h_k(n_1) = \ell^{H_k}$, so by definition of $h_k$ we obtain
  $h_k(n_1+1) = y \neq \ell^{H_k}$, a contradiction. Thus,
  $\lneg\nbox{k}_\pi\lneg S_y$, which is equivalent to $\ndia{k}_\pi S_y$.
\end{quote}

  For \cref{sprop:3}, we formalize the following argument in $T$:
\begin{quote}
  Assume $S_x$, where $x \neq 0$, and assume $\ell^{H_k} = z$. By the
  construction of the functions $h_k$, we know that
  $x \in R_{k+1}(z) \cup \{z\}$. By the definition of $\widetilde{R}_k$, this
  implies $R^\ast_k(z) \subseteq \widetilde{R}_k(x)$. Since we can define this
  property by a $\Delta_0$-formula, we know
  $\nbox{k}_\pi(R_k^\ast(z) \subseteq \widetilde{R}_k(x))$. Hence,
  $\nbox{k}_\pi(\bigvee_{y \in R^\ast_k(z)} S_y)$ implies
  $\nbox{k}_\pi (\bigvee_{y \in \widetilde{R}_k(x)} S_y).$ Moreover, since
  $\ell^{H_k} = z$, we must have $\exists n\, h(n) = z$. The latter
  statement is definable by a $\Sigma_{|\pi_k| + 1}$-formula, whence
  $\nbox{k}_\pi (\exists n\, h(n) = z)$. By~\cref{hprop:4}
  of~\Cref{lem:hprops}, we know that, for any $u \in W_0$,
  \begin{align*}
    \exists n\,h_k(n) = u\ \implies\ \ell^{H_m} \in R^\ast_k(z) \cup \{z\},
  \end{align*}
  whence
  \begin{align*}
     \nbox{k}_\pi (\exists n\,h_k(n) = u)\ \implies\ \nbox{k}_\pi(\ell^{H_m} \in R^\ast_k(z) \cup \{z\}).
  \end{align*}
  For $u = z$, we thus obtain
  $\nbox{k}_\pi(\ell^{H_m} \in R^\ast_k(z) \cup \{z\})$. Now we observe
  that $x \neq 0$ implies $z \neq 0$ and, by construction of $h_k$, we
  infer $\nbox{k}_\pi\lneg S_z$. Therefore,
  $\nbox{k}_\pi(\ell^{H_m} \in R^\ast_k(z))$, i.e.,
  $\nbox{k}_\pi(\bigvee_{y \in R^\ast_k(z)} S_y)$. We observed above
  that this implies $\nbox{k}_\pi(\bigvee_{y \in \widetilde{R}_k(x)} S_y)$, and thus
  the proof is finished.
\end{quote}

\Cref{sprop:4} can be proved by showing, using an external induction on
$k$, that $\stdmodel \models \ell^{H_k} = 0$ for all $k \geq 0$. There,
one uses the soundness of $\nbox{k}_\pi$: if $\ell^{H_k} = z \neq 0$,
then $\nbox{k}_\pi \lneg S_z$, since by induction hypothesis we have
$h_k(0) = \ell^{H_{k-1}} = 0$. Since $\nbox{k}_\pi$ is sound, it follows
that $\ell^{H_k} \neq z$. Hence, $\ell^{H_k} = 0$.
\end{proof}

Now we define an arithmetical realization $f$ by
\begin{align*}
  f \colon p \longmapsto \bigvee_{\fk{A}_0, x \models p} S_x.
\end{align*}
In~\citep{beklemishev:2011}, the following ``commutation lemma'' is
shown---recall that we fixed $\varphi$ in the beginning of our proof:
\begin{lemma}
  \label{lem:commutation}
  For every subformula $\theta$ of $\varphi$ and each
  $x \in W_0 \setminus \{0\}$:
  \begin{itemize}
  \item $\fk{A}_0, x \models \theta$ implies
    $T \proves S_x \limpl f_\pi(\theta)$;
  \item $\fk{A}_0, x \not\models \theta$ implies
    $T \proves S_x \limpl \lneg f_\pi(\theta)$.
  \end{itemize}
\end{lemma}

Using this lemma, we can conclude $T \nproves f_\pi(\varphi)$ as
follows. If we had $T \proves f_\pi(\varphi)$, then, since
$\fk{A}_0, r \not\models \varphi$, we obtain $T \proves \lneg
S_1$. Thus, $T \proves \nbox{0}_\pi S_1$ and since $0R_0 r$,
using~\ref{sprop:2}, we obtain $T \proves \lneg S_0$. By the soundness
of $T$, this implies $\stdmodel \models \lneg S_0$,
contradicting~\ref{sprop:4}. Therefore, $T \nproves f_\pi(\varphi)$ as
required.

\subsection{The Realization $f$ is Typed for $\pi$} 

We now prove, using the assumption that $\fk{A}_0$ is strongly
persistent, that $f$ is actually typed for $\pi$ which will then
conclude the arithmetical completeness proof for $\msGLP$. When
reasoning in $T$, we shall often treat $\ell^{H_i}$ ($i = 0,1,\ldots,m$)
as a world and write $\fk{A}_0, \ell^{H_i} \models p$ as an abbreviation
for the fact that, provably in $T$, $\ell^{H_i} = \ol{u}$ holds for some
$u$ such that $\fk{A}_0, u \models p$.

\begin{lemma}
    \label{lem:limitsprop}
    For all $k < m$, provably in $T$, if $k < n \leq m$, then either
    $\ell^{H_k} = \ell^{H_n}$ or $\ell^{H_k} R_j \ell^{H_n}$, for some $j \in (k, n]$.
\end{lemma}
\begin{proof}[Sketch]
  We can easily conclude from~\Cref{lem:hprops} that, for $k \geq 0$,
  either $\ell^{H_k} = \ell^{H_{k+1}}$, or
  $\ell^{H_k} R_{k + 1} \ell^{H_{k + 1}}$. Using this property, the
  claim now follows easily by an external induction on~$k$.
\end{proof}

\begin{lemma}
  \label{lem:mylemma}
  For any variable $p$ of sort $k \leq m$, provably in $T$,
  \begin{align*}
    f(p) \iff\forall w \in W_0 \setminus \val{p}\ \forall x\, \lneg H_k(x, w).
  \end{align*}
\end{lemma}
\begin{proof}
  For the direction from left to right, we reason in $T$ as
  follows. Assume $f(p)$ and, towards a contradiction, suppose that
  $\exists x\, h_k(x) = w$ for some $w \in W_0$ such that
  $\fk{A}_0, w \not\models p$. By~\cref{hprop:4} of~\Cref{lem:hprops},
  we know that, provably in $T$, $\exists x\, h_k(x) = w$ implies
  \begin{align*}
    S_w \lor \bigvee_{\mathclap{u \in R^\ast_k(w)}} S_u\rlap{.}
  \end{align*}
  Since $\fk{A}_0$ is strongly persistent and
  $\fk{A}_0, w \not\models p$, we know that $\fk{A}_0, u \not\models p$
  for all $u \in R_k^\ast(w)$. This contradicts $f(p)$
  by~\cref{sprop:1} of~\Cref{lem:sprops}.

  For the other direction, we reason in $T$ as follows. Assume the
  right-hand side of the equivalence. We certainly know that
  $\ell^{H_k} \neq u$ for all $u \in W_0$ such that
  $\fk{A}_0, u \not\models p$. Now, if $\ell^{H_k} = \ell^{H_m}$, then,
  by~\ref{sprop:1}, $S_x$ holds for some $x \in W_0$ such that
  $\fk{A}_0 ,x \models p$ and we are thus finished. So suppose that
  $\ell^{H_k} \neq \ell^{H_m}$. We know that $\fk{A}_0, \ell^{H_k} \models {p}$,
  since $\forall x\, h_k(x) \neq w$ for all $w \in W_0$ such that
  $\fk{A}_0, w \not\models {p}$.  Assume now that
  $\fk{A}_0,\ell^{H_m} \not\models p$.  By~\Cref{lem:limitsprop}
  there must be a $j \in (k, m]$ such that $\ell^{H_k} R_j
  \ell^{H_m}$.
  By strong persistence, for any $x, y \in W_0$ such that $x R_j y$, it
  holds that $\fk{A}_0, y \not\models p$ implies that
  $\fk{A}_0, x \not\models p$. Thus,
  $\fk{A}_0, \ell^{H_m} \not\models p$ is impossible and therefore
  $\fk{A}_0, \ell^{H_m} \models {p}$ by~\cref{sprop:1} of~\Cref{lem:sprops}.
\end{proof}

\begin{lemma}
  \label{lem:typed}
  For every variable $p$ of sort $k < \omega$, $f(p)$ is
  $\Pi_{|\pi_k| + 1}$ in $T$.
\end{lemma}
\begin{proof}
  Recall that, according to~\Cref{lem:hprops}, $H_k(x,y)$ is
  $\Sigma_{|\pi_k| + 1}$ in $T$.
  We remind the reader that $f(p)$ is the disjunction of all $S_x$ such
  that $\fk{A}_0,x \models p$. Observe that $S_x$ is by construction
  $\Sigma_{|\pi_m| + 2}$ in $T$ and hence so is $f(p)$. Moreover, recall
  that $m$ is the only number such that $R_m \neq \emptyset$ and
  $R_k = \emptyset$ for all $k > m$.

  Suppose first that $k > m$. Then $|\pi_m| + 2 \leq |\pi_k| + 1$ and so
  $f(p)$ is also $\Sigma_{|\pi_k| + 1}$ in $T$. Moreover,
  using~\cref{hprop:1} of~\Cref{lem:hprops}, we observe that, provably
  in $T$,
  \begin{align*}
    f(p)\ \iff\ \bigvee_{\mathclap{\fk{A}_0,x \models p}} S_x\ \iff\ \bigwedge_{\mathclap{\fk{A}_0,x \not\models p}} \lneg S_x.
  \end{align*}
  The sentences $\lneg S_x$ are $\Pi_{|\pi_k| + 1}$ in $T$ and thus
  $f(p)$ is $\Pi_{|\pi_k| + 1}$ in $T$ as well.

  Suppose now that $k \leq m$. Recall that $H_k(x,y)$ is
  $\Sigma_{|\pi_k| + 1}$ in $T$ and therefore $\lneg H_k(x,y)$ is
  $\Pi_{|\pi_k| + 1}$ in $T$.
  %
  %
   By~\Cref{lem:mylemma} we know that,
  provably in $T$,
  \begin{align*}
    f(p) \iff \forall w \in W_0 \setminus \val{p}\ \forall x\, \lneg H_k(x, w).
  \end{align*}
  Since $\lneg H_k(x,y)$ is $\Pi_{|\pi_k| + 1}$ in $T$ and since
  $\Pi_{|\pi_k|+1}$-formulas are closed under universal quantification,
  $f(p)$ is $\Pi_{|\pi_k| + 1}$ in $T$.
\end{proof}

Now~\Cref{lem:typed} implies that the realization $f$ is actually typed
for $\pi$. This concludes the proof of the arithmetical completeness
theorem (\Cref{th:msglpcompl}) for $\msGLP$.

\section{Some Further Results on $\msGLP$}
\label{sec:4}

In this section, we briefly establish some further results on $\msGLP$
that mostly rely on results previously obtained for $\GLP$.

\subsection{Truth Provability Logic} 

Let $\mathsf{GLPS}$ denote the extension of the set of theorems of GLP by the schema $\varphi \limpl \ndia{n}\varphi$, for all formulas $\varphi$ and all $n < \omega$, and with modus ponens as a sole rule of inference. It turns out that the theorems of GLPS are exactly those modal formulas that are true in the standard model of arithmetic under every arithmetical realisation 
(see~\citep{beklemishev:2011}). The methods
above can be easily extended to characterize a many-sorted analogue of
$\mathsf{GLPS}$, which we denote by $\msGLPS$. More precisely, let
$\msGLPS$ denote the logic consisting of the set of theorems of $\msGLP$
extended by the schema $\varphi \limpl \ndia{n}\varphi$ ($n \geq 0$) and
with modus ponens as its sole rule of inference. 

Let
$\ndia{n_1}\varphi_1,\ldots,\ndia{n_s}\varphi_s$ be an enumeration of
all subformulas from $\varphi$ of the form $\ndia{k}\psi$. Furthermore,
let
\begin{align*}
  U(\varphi) \coloneqq \bigwedge_{\mathclap{i = 1}}^s (\varphi_i \limpl
  \ndia{n_i}\varphi_i)\rlap{.}
\end{align*}
Then the following is a straightforward adaption of a similar result
from~\citep{beklemishev:2011} for $\mathsf{GLPS}$:

\begin{theorem}
  Let $T$ be a sound axiomatizable extension of $\PA$ and $\pi$ a strong
  sequence of provability predicates over $T$ of which every provability
  predicate is sound. Then, for all many-sorted formulas $\varphi$, the
  following statements are equivalent:
  \begin{enumerate}[label={\rm (\arabic*)}]
    \item $\msGLPS \proves \varphi$;\label{it:glps1}
    \item $\msGLP \proves U(\varphi) \limpl \varphi$; and\label{it:glps2}
    \item $\stdmodel \models f_\pi(\varphi)$, for all realizations $f$
      that are typed for $\pi$.\label{it:glps3}
  \end{enumerate}
\end{theorem}
\begin{proof}[Sketch]
  The implications from~\ref{it:glps1} to~\ref{it:glps3} and
  from~\ref{it:glps2} to~\ref{it:glps1} are easy to establish---observe
  that $\msGLPS \proves U(\varphi)$. We sketch the direction
  from~\ref{it:glps3} to~\ref{it:glps2} again by citing results
  from~\citep{beklemishev:2011}. Suppose
  $\msGLP \nproves U(\varphi) \limpl \varphi$. As in the arithmetical
  completeness proof for $\msGLP$, we can construct a finite rooted
  $\msJ$-model $\fk{A}_0$ with root $0$ such that
  $\fk{A}_0,0 \models M^+(\varphi) \land U(\varphi)$ and
  $\fk{A}_0,0 \not\models \varphi$, i.e., $\fk{A}_0$ is constructed as
  in the arithmetical completeness proof for $\msGLP$, with the only
  difference that $U(\varphi)$ is satisfied at world $0$. We can
  construct the functions $h_k$ based on $\fk{A}_0$ and the sentences
  $S_x$ in a similar spirit as in the arithmetical completeness proof of
  $\msGLP$. \Cref{lem:sprops} then holds without any changes.

  We can
  then strengthen \Cref{lem:commutation} and prove that, for every
  subformula $\theta$ of $\varphi$,
  \begin{itemize}
    \item $\fk{A}_0, 0 \models \theta$ implies $T \proves S_0 \limpl f_\pi(\theta)$;
    \item $\fk{A}_0, 0 \not\models \theta$ implies $T \proves S_0 \limpl \lneg f_\pi(\theta)$.
  \end{itemize}
  (For a proof of this result, we refer the reader
  to~\citep{beklemishev:2011}.) The proof of the fact that the
  realization $f$ is actually typed also holds without any changes. Now
  $\stdmodel \models S_0$ (\cref{sprop:4} of~\Cref{lem:sprops}) gives us
  $\stdmodel \not\models f_\pi(\varphi)$.
\end{proof}

\subsection{Reducing $\msGLP$ to $\GLP$} 

For the results contained in the
remainder of this section, we will rely on a reduction of $\msGLP$ to
$\GLP$ which we shall present here.

We first borrow some notions from~\citep{beklemishevjoost:2014} used to
reduce $\GLP$ to $\J$. Let $\varphi$ be a polymodal formula and let
$\ndia{m_1}\varphi_1,\ndia{m_2}\varphi_2,\ldots,\ndia{m_s}\varphi_s$ be
an enumeration of all subformulas of $\varphi$ of the form
$\ndia{k}\psi$ such that $i < j$ implies $m_i \leq m_j$. Define
\begin{align*}
  N(\varphi) \coloneqq \bigwedge_{\mathclap{\substack{1 \leq i \leq s \\
        i < j \leq s}}} (\ndia{m_j}\varphi_j \limpl
  \ndia{m_i}\varphi_i).
\end{align*}
Furthermore, let
\begin{align*}
  N^+(\varphi) \coloneqq N(\varphi) \land \bigwedge_{\mathclap{1 \leq i
      \leq s}}\nbox{m_i}\varphi.
\end{align*}
Notice that, if $\psi$ is a subformula of $\varphi$, then $N^+(\varphi)$
implies $N^+(\psi)$ (in any of our logics under
consideration); likewise, in this case, $N(\varphi)$ also implies
$N(\psi)$.

\begin{remark}
  The formula $N^+(\varphi)$ is reminiscent of the formula
  $M^+(\varphi)$ presented during the arithmetical completeness proof
  for $\msGLP$. However, notice that $N^+(\varphi)$ contains only
  modalities that already occur in $\varphi$ which may not be the case
  for $M^+(\varphi)$. This property will be used below.
\end{remark}

\begin{lemma}[\citep{beklemishevjoost:2014}]
  \label{lem:reductionglpj}
  For any $\psi$,
  $\GLP \proves \psi \iff \J \proves N^+(\psi) \limpl \psi$.
\end{lemma}

\begin{lemma}
  \label{lem:complhelper}
  The following are equivalent for all $\varphi$:
  \begin{enumerate}[label={\rm (\arabic*)}]
  \item $\msGLP \proves \varphi$;\label{complhelper:0}
    \item $\GLP \proves P^+(\varphi) \limpl \varphi$; \label{complhelper:1}
    \item
      $\J \proves N^+(P^+(\varphi) \limpl \varphi) \limpl (P^+(\varphi)
      \limpl \varphi)$; \label{complhelper:2}
    \item $\msJ \proves N^+(\varphi) \limpl \varphi$. \label{complhelper:3}
  \end{enumerate}
\end{lemma}
\begin{proof}
  It is clear that~\ref{complhelper:1} implies~\ref{complhelper:0} since
  $\msGLP \proves P^+(\varphi)$ and $\msGLP$ extends $\GLP$. Likewise,
  it is clear that~\ref{complhelper:3} implies~\ref{complhelper:0} since
  $\msGLP \proves N^+(\varphi)$.  The equivalence
  between~\cref{complhelper:1,complhelper:2} is the content
  of~\Cref{lem:reductionglpj}.

  We are first going to show that~\ref{complhelper:0}
  implies~\ref{complhelper:3}. Assume
  $\msJ \nproves N^+(\varphi) \limpl \varphi$. By~\Cref{lem:treemodel},
  we know there is a finite $\msJ$-model
  $\fk{A} = (W,\{R_n\}_{n \geq 0},\val{\cdot})$ with root $r$ such that
  $\fk{A}, r \not\models N^+(\varphi) \limpl \varphi$. Moreover,
  $R_n = \emptyset$ for all $n$ such that $\ndia{n}$ does not occur in
  $N^+(\varphi) \limpl \varphi$. Hence, $R_n \neq \emptyset$ implies
  that $\ndia{n}$ occurs in $\varphi$, since $N^+(\varphi)$ and
  $\varphi$ contain exactly the same modalities. Our aim is to show that
  $\fk{A},r \models M^+(\varphi)$. Consider an enumeration
  $\ndia{m_1}\varphi_1,\ndia{m_2}\varphi_2,\ldots,\ndia{m_s}\varphi_s$
  of all subformulas of $\varphi$ of the form $\ndia{k}\psi$ and let
  $n \coloneqq \max_{i \leq s} m_i$. Recall that
  \begin{align*}
    M(\varphi) &\coloneqq \bigwedge_{\mathclap{\substack{1 \leq i \leq s \\ m_i < j
    \leq n}}}(\ndia{j}\varphi_i \limpl \ndia{m_i}\varphi_i),
  \end{align*}
  and, furthermore,
  $M^+(\varphi) \coloneqq M(\varphi) \land \bigwedge_{{i \leq n}}
  \nbox{i} M(\varphi)$.
  Let $i \in \{1,\ldots,s\}$ and consider any $j$ such that
  $m_i < j \leq n$. Now $\fk{A},r \models \ndia{j}\varphi_i$ only if
  $j = m_k$ for some $k = 1,\ldots,s$. In this case,
  $\fk{A},r \models \ndia{m_i} \varphi$ since
  $\fk{A},r \models N^+(\varphi)$. Otherwise, if $j \neq m_k$ for all
  $k = 1,\ldots,s$, then trivially
  $\fk{A},r \models \ndia{j} \varphi_i \limpl \ndia{m_i}\varphi_i$,
  since $\fk{A},r \not\models \ndia{j}\varphi_i$ due to the fact that
  $R_j = \emptyset$. Let $n \coloneqq \max_{i \leq s} m_i$ and consider
  any $i \leq n$. A similar line of reasoning as before shows that
  $\fk{A},r \models \nbox{i}M(\varphi)$. Hence,
  $\fk{A},r \models M^+(\varphi)$ and so
  $\msJ \nproves M^+(\varphi) \limpl \varphi$, whence
  $\msGLP \nproves \varphi$ follows by~\Cref{th:msglpcompl}.

  To complete our proof, it remains to be shown that~\ref{complhelper:3}
  implies~\ref{complhelper:2}. Assume
  $\J \nproves N^+(P^+(\varphi) \limpl \varphi) \limpl (P^+(\varphi)
  \limpl \varphi)$.
  By~\Cref{lem:rootedjmodel}, there is a $\J$-model
  $\fk{A} = (W,\{R_n\}_{n \geq 0},\val{\cdot})$ having root $r$ such
  that $\fk{A},r \models N^+(P^+(\varphi) \limpl \varphi)$,
  $\fk{A},r \models P^+(\varphi)$, and $\fk{A},r \not\models \varphi$.
  Moreover, $\fk{A}$ is such that $R_n \neq \emptyset$ implies that
  $\ndia{n}$ occurs as a modality in
  $N^+(P^+(\varphi) \limpl \varphi) \limpl (P^+(\varphi) \limpl
  \varphi)$,
  and hence in $\varphi$. Since $\fk{A}, r \models P^+(\varphi)$, we
  conclude that $\fk{A} \models P(\varphi)$, whence
  by~\Cref{lem:persistenceR} it follows that $\fk{A}$ is strongly
  persistent and thus a $\msJ$-model. Now
  $\fk{A},r \models N^+(P^+(\varphi) \limpl \varphi)$ entails that
  $\fk{A}, r \models N^+(\varphi)$ (since $\varphi$ is a subformula of
  $P^+(\varphi) \limpl \varphi$) and so
  $\fk{A},r \not\models N^+(\varphi) \limpl \varphi$. Thus,
  $\msJ \nproves N^+(\varphi) \limpl \varphi$ by the soundness of $\msJ$
  for the class of $\msJ$-models.
\end{proof}

\Cref{lem:complhelper} in particular establishes that
$\msGLP \proves \varphi$ iff $\GLP \proves P^+(\varphi) \limpl \varphi$.
In the following, we shall use this reduction of $\msGLP$ to $\GLP$ in
order to transfer some results known for $\GLP$ to $\msGLP$.

\subsection{Craig Interpolation} 

We say that a logic $\ca{L}$ enjoys the \emph{Craig interpolation
  property} if, whenever $\ca{L} \proves \varphi \limpl \psi$, then
there is a formula $\eta$ such that
$\ca{L} \proves \varphi \limpl \eta$ and $\ca{L} \proves \eta \limpl \psi$,
and the following conditions hold:
\begin{enumerate}
\item  $\eta$ contains only variables which are present in both
$\varphi$ and $\psi$, and
\item $\eta$ has only modalities that appear in $\varphi$ or $\psi$.
\end{enumerate} 
The formula $\eta$ is called \emph{interpolant} for
$\varphi \limpl \psi$.

\begin{theorem}[\citep{ignatiev:1993,beklemishev:2007a}]
  $\GLP$ enjoys the Craig interpolation property.
\end{theorem}

\begin{remark}
  Notice we state a rather weak form of Craig interpolation, since we do
  not demand that the modalities of $\eta$ occur in both $\varphi$ and
  $\psi$. Indeed, for $\GLP$, one cannot demand that property, as the
  example $\ndia{1}p \limpl \ndia{0}p$ shows
  (cf.~\citep{beklemishev:2007a}). However, as stated in the theorem
  above, we can demand that each modality from $\eta$ is contained in
  $\varphi$ or $\psi$. We shall use this property below, when we discuss
  variants of $\msGLP$ that restrict the use of sorts and modalities.
\end{remark}

\begin{corollary}
  \label{cor:interp}
  $\msGLP$ enjoys the Craig interpolation property.
\end{corollary}
\begin{proof}
  Suppose $\msGLP \proves \varphi \limpl \psi$. Let $\Theta$ be the set
  of all modalities from $\varphi \limpl \psi$. We have
  \begin{align*}
    \GLP \proves P^+_\Theta(\varphi \limpl \psi) \limpl (\varphi \limpl
    \psi).
  \end{align*}
  Note that $P^+_\Theta(\varphi \limpl \psi)$ is equivalent in $\GLP$ to
  $P^+_\Theta(\varphi) \land P^+_\Theta(\psi)$. Hence,
  \begin{align*}
    \GLP \proves P^+_\Theta(\varphi) \land P^+_\Theta(\psi) \limpl
    (\varphi \limpl \psi),
  \end{align*}
  whence by propositional logic
  \begin{align*}
    \GLP \proves P^+_\Theta(\varphi) \land \varphi \limpl
    (P^+_\Theta(\psi) \limpl \psi).
  \end{align*}
  Since $\GLP$ enjoys the Craig interpolation property, there is an
  interpolant $\eta$ containing only variables which occur in
  $ P^+_\Theta(\varphi) \land \varphi$ and
  $P^+_\Theta(\psi) \limpl \psi$ such that
  \begin{align*}
    \GLP \proves P^+_\Theta(\varphi) \land \varphi \limpl
    \eta\quad\text{and}\quad\GLP \proves \eta \limpl (P^+_\Theta(\psi)
    \limpl \psi)\rlap{.}
  \end{align*}
  But $\msGLP \proves P^+_\Theta(\varphi)$ and $\msGLP \proves
  P^+_\Theta(\psi)$. Therefore, $\msGLP \proves \varphi \limpl \eta$ and
  $\msGLP \proves \eta \limpl \psi$. Note that $\eta$ only contains
  variables which occur in $\varphi$ and $\psi$, since
  $P^+_\Theta(\chi)$ contains exactly the variables from $\chi$, for any
  formula $\chi$.
\end{proof}

\subsection{Complexity} 

We can also
exploit the reduction of $\msGLP$ to $\GLP$ to establish a
\PSpace-completeness result for $\msGLP$.

\begin{theorem}[\citep{shapirovsky:2008}]
  Deciding whether $\GLP \proves \varphi$ is complete for \PSpace.
\end{theorem}

\begin{corollary}
  Deciding whether $\msGLP \proves \varphi$ is complete for \PSpace.
\end{corollary}
\begin{proof}
  For membership, in order to check whether $\msGLP \proves \varphi$, it
  suffices to check whether $\GLP \proves P^+(\varphi) \limpl
  \varphi$. Note that $P^+(\varphi)$ is polynomial in the size of
  $\varphi$. Indeed, let $m$ be the number of different modalities
  occurring in $\varphi$. Then the formula $P^+(\varphi)$ contains for
  each propositional variable occurring in $\varphi$ at most $m$
  conjuncts of the form $\ndia{j}p \limpl p$ and at most $m$ conjuncts
  of the form $\ndia{j}\lneg p \limpl \lneg p$. Both $m$ and the number
  of variables in $\varphi$ are clearly bounded by the size of
  $\varphi$. Hence, the size of $P^+(\varphi)$ is at most quadratic in
  the size of $\varphi$.

  For hardness, we reduce the task of checking whether
  $\GLP \proves \varphi$ to our problem as follows. Let us consider
  $\varphi$ as a many-sorted formula whose propositional variables all
  have sort $\omega$. Now we observe that
  $\GLP \proves P^+(\varphi) \limpl \varphi$ iff
  $\msGLP \proves \varphi$ and, since $\varphi$ contains only variables
  of sort $\omega$, we see that $P^+(\varphi)$ is actually $\verum$ (the
  empty conjunction), i.e., $\GLP \proves \varphi$ iff
  $\msGLP \proves \varphi$.
\end{proof}

\subsection{Omitting the Sort $\omega$}

 An interesting question is to
consider a variant of $\msGLP$ that is formulated over a language where
propositional variables only have finite sorts, that is, only sorts
$n \in \omega$. We briefly treat this case here.

We actually work in a slightly more general setting here: let
$\alpha \in \omega \cup \{\omega\}$ and let $\mGLP{\alpha}$ denote the
logic that arises from $\msGLP$ when we only allow the use of variables
of sort less than $\alpha$ and modalities $\ndia{\beta}$ with
$\beta < \alpha$. Notice that formulas in the language of
$\mGLP{\alpha}$ all have finite sort. Moreover, note that $\msGLP$
extends $\mGLP{\alpha}$ in the sense that if
$\mGLP{\alpha} \proves \varphi$, then also $\msGLP \proves
\varphi$. Furthermore, if $\beta \leq \alpha$, then $\mGLP{\alpha}$
extends $\mGLP{\beta}$. Likewise, we can also define a variant
$\mJ{\alpha}$ of $\msJ$ that enforces similar restrictions on the
language as $\mGLP{\alpha}$ does and it can be easily checked that all
the results obtained for $\msJ$ carry over to the case of $\mJ{\alpha}$.

The notion of an arithmetical realization over a strong sequence of
provability predicates immediately captures the case of formulas that
contain only variables of finite sort. The arithmetical completeness
theorem for $\mGLP{\alpha}$ then reads:
\begin{theorem}
  \label{th:mpglpcompl}
  Let $T$ be an axiomatizable extension of $\PA$ and $\pi$ a strong
  sequence of provability predicates over $T$ whose predicates are all
  sound. Let $\varphi$ be a formula in the language of
  $\mGLP{\alpha}$. The following statements are equivalent:
  \begin{enumerate}[label={\rm (\arabic*)}]
  \item $\mGLP{\alpha} \proves \varphi$;\label{mpglp:1}
  \item $T \proves f_\pi(\varphi)$, for all arithmetical realizations
    $f$ that are typed for $\pi$.\label{mpglp:2}
  \end{enumerate}
\end{theorem}
\begin{proof}[Idea]
  The direction from~\ref{mpglp:1} to~\ref{mpglp:2} is immediate by the
  arithmetical completeness theorem for $\msGLP$ (\Cref{th:msglpcompl})
  and the fact that $\msGLP$ extends $\mGLP{\alpha}$. For the other
  direction, the same construction as for $\msGLP$ can be carried
  out. Essentially, one can just ignore the case of variables of sort
  $\geq \alpha$ in the construction presented for $\msGLP$.
\end{proof}

An easy consequence of this fact is that $\msGLP$ a conservative
extension of $\mGLP{\alpha}$:
\begin{corollary}
  \label{cor:conservative}
  $\msGLP$ conservatively extends $\mGLP{\alpha}$, i.e., if $\varphi$ is
  in the language of $\mGLP{\alpha}$ and $\msGLP \proves \varphi$, then
  also $\mGLP{\alpha} \proves \varphi$. Moreover,
  $\mGLP{\beta}$ conservatively extends $\mGLP{\alpha}$, for $\beta > \alpha$.
\end{corollary}
\begin{proof}
  If $\varphi$ is in the language of $\mGLP{\alpha}$, it contains no
  variables of sort $\geq \alpha$, hence, if $\msGLP \proves \varphi$,
  then $\PA \proves f_\pi(\varphi)$ for all realizations $f$ (where
  $\pi$ is a strong sequence of provability predicates over $\PA$) that
  are typed for $\pi$. The result now follows immediately
  from~\Cref{th:mpglpcompl}.
\end{proof}

Having the above result in place, Craig interpolation for $\mGLP{\alpha}$
follows immediately:
\begin{corollary}
  $\mGLP{\alpha}$ enjoys the Craig interpolation property.
\end{corollary}
\begin{proof}
  If $\mGLP{\alpha} \proves \varphi \limpl \psi$, where $\varphi$ and
  $\psi$ contain only variables of finite sort, then
  $\msGLP \proves \varphi \limpl \psi$ and hence, by~\Cref{cor:interp},
  there is an interpolant $\eta$ such that
  $\msGLP \proves \varphi \limpl \eta$ and
  $\msGLP \proves \eta \limpl \psi$. The interpolant $\eta$ contains
  only variables that jointly appear in $\varphi$ and $\psi$. Moreover,
  each modality from the $\eta$ is contained in $\varphi$ or
  $\psi$. Hence, $\eta$ is in the language of $\mGLP{\alpha}$. Since
  $\msGLP$ conservatively extends $\mGLP{\alpha}$
  (\Cref{cor:conservative}), we obtain
  $\mGLP{\alpha} \proves \varphi \limpl \eta$ and
  $\mGLP{\alpha} \proves \eta \limpl \psi$, as desired.
\end{proof}

For the \PSpace-hardness proof of $\msGLP$, the use of variables of sort
$\omega$ become vital, and the proof thus does not immediately carry
over to the case of $\mGLP{\alpha}$. We thus aim at a different proof in
the following, essentially exploiting a reduction of the intuitionistic
propositional calculus (henceforth denoted $\IPC$) to the standard
G\"odel-L\"ob logic $\GL$.\footnote{$\GL$ can be axiomatized by axiom
  schemas~\ref{glp:scheme:taut} to~\ref{glp:scheme:2} of $\GLP$ (with
  $\ndia{n}$ replaced by $\Diamond$) and is closed under modus ponens
  and $\varphi\limpl \psi/\Diamond\varphi \limpl \Diamond\psi$; see,
  e.g.,~\citep{boolos:1995} for an extensive treatment of $\GL$.} For
details on $\IPC$ and its translation to $\GL$ we refer the interested
reader to~\citep{dejongh:2006}.

The translation $\cdot^\ast$ of formulas from $\IPC$ to formulas of
$\GL$ is defined as follows:
  \begin{itemize}
    \item $\falsum^\ast \coloneqq \falsum$;
    \item $p^\ast \coloneqq \Box p \land p$, where $p$ is a propositional variable;
    \item $(\varphi \limpl \psi)^\ast \coloneqq \Box(\varphi^\ast \limpl \psi^\ast) \land (\varphi^\ast \limpl \psi^\ast)$;
    \item $(\varphi \land \psi)^\ast \coloneqq \varphi^\ast \land \psi^\ast$;
    \item $(\varphi \lor \psi)^\ast \coloneqq \varphi^\ast \lor \psi^\ast$.
  \end{itemize}

  Let $\fk{A} = (W, R, \val{\cdot})$ be a Kripke model. We say that
  $\fk{A}$ is \emph{reversely persistent}, if for any variable $p$ and
  all $x, y \in W$, it holds that $\fk{A},x \models p$ and $xRy$ imply
  $\fk{A},y \models p$.\footnote{We call this property reversely
    persistence here, since ``persistence'' in this paper refers to
    propagation of truth-values in the other direction. However, we
    remark that ``reverse persistence'' is usually called
    ``persistence'' (also in~\citep{dejongh:2006}).} We say that
  $\fk{A}$ is an \emph{intuitionistic} Kripke model, if it is reversely
  persistent and $R$ is reflexive and transitive. The \emph{irreflexive
    version of $\fk{A}$} is the model
  $\fk{A}^\ast \coloneqq (W^\ast, R^\ast,\val{\cdot})$ with
  $W^\ast \coloneqq W$, $x R^\ast y\, \defequiv\, xRy$ and $x \neq y$,
  and $\fk{A}^\ast,x \models p\; \defequiv\; \fk{A},x \models p$, for
  all variables $p$.

\begin{lemma}[\citep{dejongh:2006}]
  \label{ipctogl}
  $\IPC \proves \varphi$ iff $\GL \proves \varphi^\ast$. Moreover, for
  any finite intuitionistic Kripke model $\fk{A}$, it holds that
  $\fk{A},x \models \varphi$ iff $\fk{A}^\ast,x \models \varphi^\ast$,
  where $\fk{A}^\ast$ is the irreflexive version of $\fk{A}$.
\end{lemma}

We are going to use the rather
well-known result that deciding whether $\IPC \proves \varphi$ is complete
for \PSpace:

\begin{theorem}[\citep{dejongh:2006}]
  \label{ipcpspace}
  Deciding whether $\IPC \proves \varphi$ is complete for \PSpace.
\end{theorem}

Recall that $\mGLP{1}$ is the fragment of $\msGLP$ that is formulated
over variables of sort $0$ and only uses the modality $\ndia{0}$, which
we abbreviate by $\Diamond$ in the following (likewise we write $\Box$
for $\nbox{0}$). We aim to show that deciding whether
$\mGLP{1} \proves \varphi$ is already hard for \PSpace.

Towards this end, we are going to take an intermediate step and prove
that Visser's $\Sigma_1$-logic (see~\citep{boolos:1995,visser:1981})
$\GLV$ is \PSpace-complete.  The logic $\GLV$ consists of all theorems
of $\GL$ plus the axioms $p \limpl \Box p$, where $p$ is a propositional
variable. $\GLV$ is closely related to $\mGLP{1}$ with the difference
that $\GLV$ is arithmetically complete for the interpretation that
assigns $\Sigma_1$-sentences to propositional variables rather than
$\Pi_1$-sentences as in the case of $\mGLP{1}$ (and $\Box$ being
interpreted as the standard G\"odelian provability predicate);
cf.~\citep{visser:1981,boolos:1995}.  Notice that any reversely
persistent model validates the axioms of the form $p \limpl \Box
p$. Moreover:

\begin{theorem}[\citep{visser:1981}]
  \label{glvsound}
  $\GLV$ is sound and complete for the class of finite, irreflexive,
  transitive, and reversely persistent Kripke models.
\end{theorem}

\begin{lemma}
  \label{lem:hardnesshelp1}
  Deciding whether $\GLV \proves \varphi$ is hard for \PSpace.
\end{lemma}
\begin{proof}

  Consider a formula $\varphi$ in the language of $\IPC$ and its
  translation $\varphi^\ast$. We claim that $\IPC \proves \varphi$ iff
  $\GLV \proves \varphi^\ast$ which will then prove the claim of the
  lemma by virtue of \Cref{ipcpspace}. 

  Indeed, if $\IPC \proves \varphi$ then $\GL \proves \varphi^\ast$,
  whence $\GLV \proves \varphi^\ast$ since $\GLV$ clearly extends
  $\GL$. On the other hand, if $\IPC \nproves \varphi$, then there is a
  finite intuitionistic Kripke model $\fk{A} = (W, R, \val{\cdot})$ such
  that $\fk{A},x \not\models \varphi$ for some $x \in W$;
  see~\citep{dejongh:2006}. Let $\fk{A}^\ast$ be the irreflexive version
  of $\fk{A}$. By \Cref{ipctogl} we also have
  $\fk{A}^\ast, x \not\models \varphi^\ast$. Since $\fk{A}^\ast$ is an
  irreflexive and transitive model, it validates all theorems of
  $\GL$. Moreover, since it is reversely persistent, it also satisfies
  the axioms $p \limpl \Box p$. By \Cref{glvsound} we thus obtain
  $\GLV \nproves \varphi^\ast$ as required.
\end{proof}

\begin{lemma}
  \label{lem:hardnesshelp2}
  Deciding whether $\mGLP{1} \proves \varphi$ is hard for \PSpace.
\end{lemma}
\begin{proof}
  Let $\cdot^\lneg$ denote the translation from formulas in the language
  of $\mGLP{1}$ to formulas of $\GLV$ that replaces each propositional
  variable $p$ by its negation $\lneg p$. We claim that
  $\mGLP{1} \proves \varphi$ iff $\GLV \proves \varphi^\lneg$. As
  mentioned above, in~\citep{visser:1981} it is shown that $\GLV$ is
  arithmetically complete for \emph{$\Sigma_1$-realizations}, i.e.,
  arithmetical realizations that assign $\Sigma_1$-sentences to
  propositional variables. By~\Cref{th:mpglpcompl}, we know that
  $\mGLP{1}$ is arithmetically complete for arithmetical realizations
  that assign $\Pi_1$-sentences to propositional variables. (We choose a
  strong sequence of provability predicates $\pi$ that has the standard
  G\"odelian predicate $\Box_\PA$ as its $0$-th predicate.)  Now every
  $\Pi_1$-sentence ($\Sigma_1$-sentence, respectively) is equivalent to
  the negation of a $\Sigma_1$-sentence ($\Pi_1$-sentence,
  respectively). Hence, the result follows immediately
  by~\Cref{lem:hardnesshelp1} and by applying the respective
  arithmetical completeness theorems for $\mGLP{1}$ and $\GLV$.
\end{proof}

\begin{theorem}
  For any $\alpha \in \omega \cup \{\omega\}$, deciding whether
  $\mGLP{\alpha} \proves \varphi$ is complete for \PSpace.
\end{theorem}
\begin{proof}
  For membership, we observe that deciding
  $\mGLP{\alpha} \proves \varphi$ amounts to deciding
  $\msGLP \proves \varphi$, since $\GLP^\ast$ conservatively extends
  $\mGLP{\alpha}$ by~\Cref{cor:conservative}. For hardness, notice that
  checking $\mGLP{1} \proves \varphi$ (where $\varphi$ is in the
  language of $\mGLP{1}$) can be reduced to
  $\mGLP{\alpha} \proves \varphi$, again by~\Cref{cor:conservative}. The
  problem of deciding whether $\mGLP{1} \proves \varphi$ is hard for
  \PSpace\ by~\Cref{lem:hardnesshelp2}, whence the claim follows.
\end{proof}

\section{A Positive Variant of $\msGLP$}
\label{sec:many}

In this section we are going to study a positive variant of $\msGLP$
whose one-sorted counterpart has been studied recently
in~\citep{beklemishev:2014}. It was noticed in~\citep{beklemishev:2012}
that the proof-theoretic analysis of Peano arithmetic in the framework
of $\GLP$ only relies on certain positive formulas. This fragment,
denoted by $\RC$, is much simpler than $\GLP$, yet expressive enough for
major proof-theoretic applications of $\GLP$ as carried out
in~\citep{beklemishev:2004,beklemishev:2005}. In particular, $\RC$
allows one to define a system of ordinal notations up
to~$\varepsilon_0$.

Formulas of $\RC$ are implications of the form $A \gives B$ (called
\emph{sequents}), where $A$ and $B$ are \emph{positive formulas}
constructed using $\land$, $\verum$, diamond modalities $\ndia{n}$, and
propositional variables only. $\RC$ and fragments thereof were
axiomatized in~\citep{dashkov:2012}, where it is also proved that, in
contrast to $\GLP$, $\RC$ is complete for a natural class of finite
Kripke frames and that theoremhood for $\RC$ is decidable in polynomial
time.

Apart from its convenient computational properties, $\RC$ also allows
for a more general arithmetical interpretation than that of standard
$\GLP$. In~\citep{beklemishev:2014}, the second author of this paper
considers an arithmetical interpretation of positive formulas where
propositional variables are interpreted as (primitive recursive
enumerations of) arithmetical theories rather than single
sentences. This allows one to interpret the diamond modalities as
\emph{reflection schemas}, which are generalizations of consistency
assertions and are not necessarily finitely axiomatizable (therefore,
in~\citep{beklemishev:2014}, positive fragments of $\GLP$ are coined
\emph{reflection calculi}). In particular, the full uniform reflection
principle is realized in~\citep{beklemishev:2014} as a modality
$\ndia{\omega}$ that is part of the calculus $\RC\omega$ which
essentially extends $\RC$ to capture this modality.

Apart from the richer interpretation of the standard diamond modalities,
the fact that variables can be interpreted as arithmetical theories
allows the introduction of additional modalities that have no
counterpart in standard $\GLP$. To wit, $\RC$ has recently been extended
in~\citep{beklemishev:2017} in order to capture modalities that express partial
conservativity operators.

Considering our introduction of many-sorted $\GLP$, it is natural to ask
whether many-sorted logics make sense in the positive setting as
well. Therefore, in this section, we introduce a many-sorted variant of
the reflection calculus $\RC\omega$ presented
in~\citep{beklemishev:2014} and prove that our calculus is
arithmetically complete. In the arithmetical completeness proof, we rely
on the construction presented in~\citep{beklemishev:2014} for the
one-sorted setting.

\subsection{Basics}

We shall consider (many-sorted) \emph{positive formulas} that are formed
using propositional variables (again having sorts up to $\omega$ as in
the setting of $\msGLP$), conjunction ($\land$), the truth constant
$\verum$, and the diamond modalities $\ndia{\alpha}$, where $\alpha$ is
either a natural number or $\omega$. We shall write $\alpha A$ instead
of $\ndia{\alpha}A$ in the following. A \emph{sequent} is an expression
of the form $A \gives B$, where $A$ and $B$ are positive formulas---the
sequent $A \gives B$ stands for the formula $A \limpl B$.  The notion of
\emph{sort} is defined in the positive setting in exactly the same way
as it is defined for the more general $\GLP$. As before, the sort of $A$
is denoted by $\sort{A}$.

The following axiom schemas and rules of inference are propositional
ones and serve as a basis for the calculi to be presented:
\begin{enumerate}
  \item $A \gives A$;\quad $A \gives \verum$;\label{ax:pos1}
  \item $A \land B \gives A$; \quad $A \land B \gives B$;\label{ax:pos2}
  \item if $A \gives B$ and $B \gives C$, then infer $A \gives C$;\label{ax:pos3}
  \item if $A \gives B$ and $A \gives C$, then infer $A \gives B \land C$.\label{ax:pos4}
\end{enumerate}
Apart from these propositional axiom schemas and rules, our calculi will
all be closed under the following rule that essentially amounts to the
necessitation rule for standard modal logics:
\begin{enumerate}[resume]
\item if $A \gives B$ then infer $\alpha A \gives \alpha B$, for any
  $\alpha \leq \omega$.\label{ax:pos5}
\end{enumerate}

The positive logic $\msRCo$ is axiomatized by the schemas and
rules~\ref{ax:pos1} to~\ref{ax:pos5} as well as the following axiom
schemas:
\begin{enumerate}[resume]
  \item $\alpha A \gives A$, whenever $\sort{A} \leq \alpha$ (\emph{$\alpha$-persistence});\label{ax:pos6}
  \item $\alpha A \gives \beta A$, for $\beta < \alpha$ (\emph{monotonicity});\label{ax:pos7}
  \item $\alpha A \land B \gives \alpha(A \land B)$, where
    $\sort{B} < \alpha$.\label{ax:pos8}
\end{enumerate}

\begin{remark}
  It is worth commenting briefly on the axiomatization of $\msRCo$. The
  calculus presented here is a many-sorted version of the calculus
  $\RC\omega$ from~\citep{beklemishev:2014}. Essentially, in
  $\RC\omega$, the axiom schema of $\alpha$-persistence can only be
  applied to the case $\alpha = \omega$. Moreover, in $\RC\omega$, the
  axiom schema~\ref{ax:pos8} is replaced by
  \begin{itemize}
  \item
    $\alpha A \land \beta B \gives \alpha(A \land \beta B)$, for
    $\beta < \alpha$.
  \end{itemize}
  It is immediate that $\msRCo$ extends $\RC\omega$ in the same sense as
  $\msGLP$ extends $\GLP$.  

  The axiom schema of $\alpha$-persistence
  (schema~\ref{ax:pos6}) is essentially $\Sigma_{\alpha+1}$-completeness
  in the setting of $\msGLP$. Unlike $\msGLP$ and $\msJ$, $\msRCo$ has
  another axiom schema that refers to the notion of sorts, namely
  schema~\ref{ax:pos8}. This is due to the lack of negation in the
  positive calculi. Indeed, suppose $\sort{B} < n < \omega$. Then
  $\msJ \proves \lneg B \limpl \ndia{n}\lneg B$, whence
  $\msJ \proves B \limpl \nbox{n} B$, and so
\begin{align*}
  \msJ \proves \ndia{n} A \land B &\limpl \nbox{n} B\\
  &\limpl \ndia{n}(A \land B),
\end{align*}
by standard modal reasoning. That is, modulo the modality
$\ndia{\omega}$, axiom schema~\ref{ax:pos8} is readily derived in $\msJ$
and thus in $\msGLP$.
\end{remark}

The notion of a proof in $\msRCo$ is defined in the expected manner and
theoremhood is denoted by $\msRCo \proves A \gives B$.  For a set
$\Gamma$ of positive formulas, we shall write
$\msRCo \proves \Gamma \gives A$ if there are
$B_1,\ldots,B_n \in \Gamma$ such that
$\msRCo \proves B_1 \land \cdots \land B_n \gives A$. We denote by
$B(p/A)$ the result of substituting the variable $p$ by the positive
formula $A$ in $B$. Substitutions in this logic must again respect the
sorts of variables. We then have:

\begin{lemma}
  \label{lem:rcsubst}
  Suppose $\msRCo \proves A \gives B$ and
  $\sort{A},\sort{B} \leq \alpha$. Then
  $\msRCo \proves C(p/A) \gives C(p/B)$ for any $C$, where
  $\sort{p} \geq \alpha$.
\end{lemma}
\begin{proof}
  By an easy induction on the structure of $C$.
\end{proof}

\subsection{Arithmetical Interpretation}  
The arithmetical interpretation
for the positive calculi presented in~\citep{beklemishev:2014} assigns
\emph{primitive recursive numerations} of theories extending $\PA$ to
propositional variables. We shall adapt this interpretation to the
many-sorted setting in the following.

Recall that, in the setting of $\msGLP$, one admissible interpretation
of the modality $\ndia{n}$ is that of \emph{$n$-consistency}, i.e.,
consistency in $\PA$ plus the set of all true $\Pi_{n+1}$-sentences. Also
recall that we denote by $\nbox{n}_\PA(x)$ and
$\ndia{n}_\PA(x)$ arithmetical formulas that respectively express
$n$-provability and $n$-consistency in $\PA$; cf.~\Cref{sec:2}.  The
arithmetical interpretation of positive formulas in the language of
$\msRCo$ is generalized in two ways:
\begin{enumerate}[label=(\arabic*)]
\item Propositional variables are interpreted as arithmetical theories
  extending $\PA$ rather than sentences. These theories are formally
  presented by a bounded formula $\sigma \coloneqq \sigma(x)$ that
  arithmetically defines the set of axioms of the theory at hand.
\item Diamond modalities are interpreted as generalized consistency
  assertions, namely, reflection principles for theories extending
  $\PA$. The modality $\ndia{\omega}$ is interpreted as the full uniform
  reflection principle that has no finite axiomatization.
\end{enumerate}
We are going to formalize these two notions in the following.

A (\emph{primitive recursive})\emph{ numeration} is a bounded formula
$\sigma \coloneqq \sigma(x)$ which defines the G\"odel numbers of the
axioms of an extension $S$ of $\PA$. We say that $\sigma$
\emph{numerates} $S$. Furthermore, we say that $\sigma$ \emph{numerates a
  $\Pi_{n+1}$-axiomatized extension of $\PA$} if
\begin{align*}
  \PA \proves \forall \alpha\,(\sigma(\alpha) \limpl \Ax_\PA(\alpha) \lor \alpha \in \Pi_{n+1}),
\end{align*}
where the expression ``$\alpha \in \Pi_{n+1}$'' denotes a natural
bounded formula which expresses that $\alpha$ is the G\"odel number of a
$\Pi_{n+1}$-sentence (possibly using $n$ as an additional parameter) and
$\Ax_\PA(\alpha)$ is a formula defining the G\"odel numbers of the
axioms of $\PA$.\footnote{Recall our convention that Greek letters
  $\alpha,\beta,\ldots$ occurring in arithmetical formulas range over
  codes of formulas.} Thus, in case $\sigma$ numerates a
$\Pi_{n+1}$-axiomatized extension of $\PA$, $\sigma(x)$ provably defines
the set of axioms of a theory that is an extension of $\PA$ by a set of
$\Pi_{n+1}$-sentences.

For a numeration $\sigma$, we denote by $\Box_\sigma(\alpha)$ the
formula which defines the standard provability predicate of the theory
numerated by $\sigma$. For numerations $\sigma$ and $\tau$, we write
$\sigma \gives_\PA \tau$ if
\begin{align*}
  \PA \proves \forall \alpha\,(\Box_\tau(\alpha) \limpl
  \Box_\sigma(\alpha))\rlap{,}
\end{align*}
and we write $\sigma \gives \tau$ if
\begin{align*}
  \stdmodel \models \forall \alpha\,(\Box_\tau(\alpha) \limpl
  \Box_\sigma(\alpha)).
\end{align*}
We assume that every numeration, provably in $\PA$, numerates an
extension of $\PA$, that is, $\tau \gives_\PA \Ax_\PA$, for any $\tau$. As
usual, we write $\Box_\sigma \varphi$ instead of
$\Box_\sigma(\code{\varphi})$ if no confusion arises. We denote by
$\Con(\sigma)$ the sentence $\lneg\Box_\sigma\falsum$.

The formula $\Con_n(\sigma)$ expresses that the theory numerated by
$\sigma$ is $n$-consistent. We often regard $\code{\Con_n(\sigma)}$ as a
definable term which depends on $n$ and use that fact without adhering
to any special notation.

Now let $\sigma$ numerate $S$. The formula $\Con_n(\sigma)$ is another
way of expressing to the so-called \emph{global $\Pi_{n+1}$-reflection
  principle for $S$}; see, e.g.,~\citep{beklemishev:2005}.  When proving
statements about $\Con_n(\sigma)$, we shall in the following often use
the following equivalent characterization without any further comment:

\begin{lemma}[\citep{beklemishev:2005}]
  \label{lem:connequivform}
  For all $n \in \omega$, $\Con_n(\sigma)$ is provably equivalent in $\PA$
  to
  \begin{align*}
    \forall \alpha \in \Pi_{n+1}\,(\Box_\sigma(\alpha) \limpl
    \True_{\Pi_{n+1}}(\alpha))\rlap{.}
  \end{align*}
\end{lemma}

Given any arithmetical sentence $\varphi$, we denote by $\ul{\varphi}$
the numeration
\begin{align*}
  \Ax_\PA(\alpha) \lor \alpha = \code{\varphi}\rlap{,}
\end{align*}
which numerates the theory $\PA + \varphi$. In this setting, for any
numeration $\sigma$, $\Conn_n(\sigma)$ numerates the theory $\PA +
\Con_n(\sigma)$. 
The schema
\begin{align*}
  \Con_\omega(\sigma)\colon\quad\{\Con_n(\sigma) \mid n \in \omega\}
\end{align*}
is well-known to be equivalent over $\PA$ to the \emph{full uniform
  reflection principle for $S$}, see, e.g.,~\citep{beklemishev:2005}. We
shall denote by $\Conn_\omega(\sigma)$ a numeration which numerates the
theory $\PA + \Con_\omega(\sigma)$.

We are now ready to formally specify the intended arithmetical
interpretation of $\msRCo$:

\begin{definition}
  \label{definition:positive:arithmetical:realization}
  An \emph{arithmetical realization} is a function $f$ from
  positive formulas to numerations such that the following conditions
  are satisfied:
  \begin{itemize}
    \item $f(\verum) = \Ax_\PA$;
    \item $f(A \land B) = f(A) \lor f(B)$;
    \item $f(\alpha A) = \Conn_\alpha (f(A))$, for $\alpha \leq \omega$.
  \end{itemize}
  We say that $f$ is \emph{typed}, if the following condition is satisfied:
  \begin{itemize}
  \item for every propositional variable $p$ of sort $\alpha$, $f(p)$
    is a numeration which numerates
      \begin{enumerate*}[label=(\arabic*)]
      \item a $\Pi_{\alpha+1}$-axiomatized extension of $\PA$ in
        case $\alpha < \omega$ and
      \item an arbitrary extension of $\PA$ in case $\alpha = \omega$;
      \end{enumerate*}
  \end{itemize}
\end{definition}



\begin{lemma}
  Let $f$ be a typed arithmetical realization and $A$ a formula
  such that $\sort{A} < \omega$. Then $f(A)$ numerates a
  $\Pi_{\sort{A} + 1}$-axiomatized extension of $\PA$.
\end{lemma}
\begin{proof}
  By an easy induction on $A$. The cases for propositional variables and
  $\verum$ are clear. For the induction step, notice that for $n <
  \omega$, $\Con_n(\sigma)$ provably belongs to $\Pi_{n+1}$, for any
  numeration $\sigma$. Furthermore, provably in $\PA$, if $\varphi$
  belongs to $\Pi_{m}$, then also to $\Pi_{n}$, for $n > m$. Using these
  facts, the claim easily follows.
\end{proof}

\begin{lemma}[\citep{beklemishev:2014}]
  \label{lem:numerate}
  Let $\sigma$ numerate $S$ and $\varphi$ be a $\Pi_{n+1}$-sentence. If
  $S \proves \varphi$ then $\PA + \Con_n(\sigma) \proves
  \varphi$. Moreover, this statement is formalizable uniformly
  in $n$ in $\PA$, i.e.,
  \begin{align*}
    \PA \proves \forall n\,\forall \alpha \in \Pi_{n+1}\,
    (\Box_\sigma(\alpha) \limpl \Box_{\Conn_n(\sigma)}(\alpha))\rlap{.}
  \end{align*}
\end{lemma}

\begin{lemma}
  Let $\sigma$ be a numeration and $n < \omega$. Then $\Conn_n(\sigma)
  \gives_\PA \sigma$, whenever $\sigma$ numerates a
  $\Pi_{n+1}$-axiomatized extension of $\PA$.
\end{lemma}
\begin{proof}
  We reason in $\PA$ as follows. Suppose $\Box_\sigma(\varphi)$ and
  reason by induction on proof length of $\varphi$. The only interesting
  case is when $\varphi \in \Pi_{n+1}$ is an
  axiom. By~\Cref{lem:numerate}, we obtain
  $\Box_{\Conn_n(\sigma)}(\varphi)$. Hence,
  $\Conn_n(\sigma) \gives_\PA \sigma$ as required.
\end{proof}

\begin{lemma}[\citep{beklemishev:2014}]
  For any numeration $\sigma$, $\Conn_\omega(\sigma) \gives_\PA \sigma$.
\end{lemma}

\begin{lemma}
  Let $\varphi$ be a $\Pi_{m+1}$-sentence and $\sigma$ a numeration. For
  $m < n < \omega$ it holds that
  \begin{align*}
  \PA \proves \Con_n(\sigma) \land \varphi \limpl
  \Con_n(\sigma \lor \ul{\varphi})\rlap{.}
  \end{align*}
\end{lemma}
\begin{proof}
  We reason in $\PA$ as follows. Suppose
  $\Box_{\sigma \lor \ul{\varphi}}(\psi)$ for $\psi \in \Pi_{n + 1}$.
  Then $\Box_\sigma(\varphi \limpl \psi)$ by a formalized version of the
  standard deduction theorem. We know that $\varphi \limpl \psi$ is a
  $\Pi_{n+1}$-sentence since $m < n$. Thus, if $\Con_n(\sigma)$ then
  also $\True_{\Pi_{n+1}}(\varphi \limpl \psi)$ and so
  $\True_{\Pi_{n+1}}(\varphi) \limpl \True_{\Pi_{n+1}}(\psi)$. Now if
  $\varphi$ holds, then, since $\varphi \in \Pi_{n+1}$, we obtain
  $\True_{\Pi_{n+1}}(\varphi)$ whence $\True_{\Pi_{n+1}}(\psi)$ follows
  as required.
\end{proof}

\begin{lemma}
  Suppose $\tau$ numerates a $\Pi_{m+1}$-axiomatized extension of $\PA$. Then,
  for any numeration $\sigma$,
  \begin{align*}
    \Conn_\omega(\sigma) \lor \tau \gives_\PA \Conn_\omega(\sigma \lor
    \tau)\rlap{.}
  \end{align*}
\end{lemma}
\begin{proof}[Sketch]
  We show an informal version of this statement by an argument
  formalizable in $\PA$. That is, we must show that for each $n$,
  \begin{align*}
    \PA + \Conn_\omega(\sigma) + \tau \proves \Conn_n(\sigma \lor \tau).
  \end{align*}
  We may assume $n > m$ and use the previous lemma. A formalization of
  the corresponding argument yields the proof.
\end{proof}

\begin{proposition}
  \label{prop:rcasound}
  $\msRCo$ is arithmetically sound, i.e., if $\msRCo \proves A \gives B$
  then $f(A) \gives_\PA f(B)$ for every typed arithmetical
  realization $f$.
\end{proposition}
\begin{proof}
  By induction on the length of a derivation of $A \gives B$. The
  soundness of the propositional rules and axioms (i.e.,~\ref{ax:pos1}
  to~\ref{ax:pos4}) are immediate. The soundness of the modal axiom
  schemas~\ref{ax:pos6},~\ref{ax:pos7}, and~\ref{ax:pos8} follows from
  the previous lemmas and corollaries. For the monotonicity axiom
  schema~\ref{ax:pos7}, it is clear that
  $\Conn_\alpha(\sigma) \gives_\PA \Conn_\beta(\sigma)$, for
  $\alpha > \beta$, since the strength of $\Con_\alpha(\sigma)$
  increases with $\alpha$.

  It remains to be shown that the necessitation rule~\ref{ax:pos5} is
  sound. Suppose $f(A) \gives_\PA f(B)$ and let $n < \omega$. We claim
  that $\PA + \Con_n(f(A)) \proves \Con_n(f(B))$. Indeed, reasoning in
  $\PA + \Con_n(f(A))$, we see that if $\varphi \in \Pi_{n+1}$ and
  $\Box_{f(B)}(\varphi)$ holds, then also $\Box_{f(A)}(\varphi)$ (since
  $f(A) \gives_\PA f(B)$) and thus also $\True_{\Pi_{n+1}}(\varphi)$. By
  \Cref{lem:connequivform}, we thus obtain
  $\PA + \Con_n(f(A)) \proves \Con_n(f(B))$, i.e.,
  $\Conn_n(f(A)) \gives_\PA \Conn_n(f(B))$.
  
  Formalizing this argument also establishes that if
  $f(A) \gives_\PA f(B)$, then
  $\Conn_\omega(f(A)) \gives_\PA \Conn_\omega(f(B))$.
\end{proof}

\subsection{Arithmetical Completeness}

The arithmetical completeness for $\msRCo$ is obtained in a similar
fashion as the results for $\msGLP$ are obtained from the arithmetical
completeness proof of $\GLP$. To obtain arithmetical completeness for
$\msRCo$, one follows the proof for $\RC\omega$ as given
in~\citep{beklemishev:2014}.

Arithmetical completeness for $\msRCo$ can thus be roughly obtained as
follows:
\begin{itemize}
\item One identifies a class of Kripke models for which $\msRCo$ is
  sound and complete and which reflects the notion of sort in an
  appropriate way. It turns out that, as in the case of $\msGLP$, the
  notion of strong persistence is appropriate for this purpose. 
\item The arithmetical completeness of $\msRCo$ is established following
  the completeness proof for $\RC\omega$ as presented
  in~\citep{beklemishev:2014}. One exploits the fact that sequents that
  are non-provable in $\msRCo$ have Kripke counterexamples that are
  strongly persistent and observes that redoing the construction
  of~\citep{beklemishev:2014} admits the extraction of an arithmetical
  counterexample that is actually typed. Notice that this is in the same
  spirit as we conducted the arithmetical completeness proof for
  $\msGLP$---after all, it was enough to observe that the assumption of
  having a strongly persistent counterexample at hand allows one to
  conclude that the arithmetical realization constructed in the
  proof for standard $\GLP$ is already typed.
\end{itemize}
In the following, we shall elaborate on the arithmetical completeness
proof for $\msRCo$.

\subsubsection{Kripke Models}

We require an appropriate class of Kripke models for which
$\msRCo$ is complete. Let $\Phi$ be a set of positive formulas and
$$\ell(\Phi) \coloneqq \{\alpha \leq \omega \mid \text{$\alpha$ occurs in
  some $A \in \Phi$}\}.$$ We say that $\Phi$ is \emph{adequate}, if it is
closed under subformulas, $\verum \in \Phi$, and
\begin{enumerate}
\item if $\beta A \in \Phi$ and $\beta < \alpha \in \ell(\Phi)$, then
  $\alpha A \in \Phi$;
\item for any variable $p$ of sort $\alpha$, if $p \in \Phi$, then
  $\beta p \in \Phi$, for all $\beta \leq \alpha$.
\end{enumerate}
An \emph{$\msRCo$-theory in $\Phi$} is a set $\Gamma \subseteq \Phi$
such that $\msRCo \proves \Gamma \gives A$ and $A \in \Phi$ implies
$A \in \Gamma$.

The notion of a Kripke model immediately extends to positive formulas as
well once we include an accessibility relation $R_\omega$, i.e., Kripke
models are structures of the form
$\fk{A} = (W, \{R_\alpha\}_{\alpha \leq \omega}, \val{\cdot})$. Recall
that $A \gives B$ stands for $A \limpl B$ and hence we specify
$\fk{A},x \models A \gives B$ iff $\fk{A},x \models A \limpl B$.  The
notion of validity in a model thus immediately extends to sequents as
well. Moreover, the notions of $\J$-model and $\msJ$-model then carry
over to the positive case by additionally considering the relation
$R_\omega$. Recall that a $\msJ$-model is a strongly persistent
$\J$-model, and that a strongly persistent model $\fk{A}$ satisfies the
following conditions, for all $0 \leq \alpha \leq \omega$:
\begin{enumerate}[label={(\arabic*)}]
\item if $\sort{p} \leq \alpha$ and $\fk{A}, y \models p$, then
  $\fk{A}, x \models p$ whenever $x R_\alpha y$; and
\item if $\sort{p} < \alpha$ and $\fk{A}, y \not\models p$, then
  $\fk{A}, x \not\models p$ whenever $x R_\alpha y$.
\end{enumerate}
In particular, for the case $\alpha = \omega$, the first condition
states that the satisfaction of \emph{any} variable is propagated
downwards along $R_\omega$-arcs, since all variables have sort at most
$\omega$.

Let $\Phi$ be an adequate set. We say that a model $\fk{A}$ is
\emph{$\Phi$-monotone}, if for any $\alpha A \in \Phi$ and
$\beta \in \ell(\Phi)$ such that $\alpha < \beta$,
$\fk{A}, x \models \beta A$ implies $\fk{A},x \models \alpha A$. The
following completeness result for $\msRCo$ is an almost literal
repetition of a similar result for $\RC\omega$ proven
in~\citep{beklemishev:2014}. We omit a proof of this theorem, since it
can be proved by a straightforward adaption of the according result
in~\citep{beklemishev:2014}.

\begin{theorem}
  \label{th:rcsemantics}
  Let $\Phi$ be a finite adequate set. Then there is a finite model
  $\fk{A} = (W, \{R_\alpha\}_{\alpha \leq \omega},\val{\cdot})$ such that
  \begin{enumerate}[label={\rm (\arabic*)}]
  \item $\fk{A}$ is an irreflexive $\msJ$-model, i.e., a $\msJ$-model in
    which all $R_\alpha$ are irreflexive;
  \item $R_\alpha = \emptyset$, for all $\alpha \not\in \ell(\Phi)$;
  \item $\fk{A}$ is $\Phi$-monotone;
  \item for any $\msRCo$-theory $\Gamma$ in $\Phi$, there is a node
    $x \in W$ such that, for any formula $A$, $A \in \Gamma$ iff
    $\fk{A},x \models A$.
  \end{enumerate}
\end{theorem}

\subsubsection{Arithmetical Completeness for $\msRCo$} We are now going to
prove the arithmetical completeness theorem for $\msRCo$, relying on the
construction for $\RC\omega$ presented in~\citep{beklemishev:2014}:
\begin{theorem}\label{arithmet-compl}
  The following are equivalent:
  \begin{enumerate}[label={\rm(\arabic*)}]
  \item $\msRCo \proves A \gives B$;\label{rccompl:1}
  \item $f(A) \gives_\PA f(B)$, for every typed arithmetical
    realization $f$;\label{rccompl:2}
  \item $f(A) \gives f(B)$, for every typed arithmetical realization
    $f$.\label{rccompl:3}
  \end{enumerate}
\end{theorem}
Note that the implication from~\ref{rccompl:1} to~\ref{rccompl:2} was
proved in~\Cref{prop:rcasound} and statement~\ref{rccompl:2} clearly
implies~\ref{rccompl:3}. In what follows, we establish that
~\ref{rccompl:3} implies~\ref{rccompl:1}. We do so by proving its
contrapositive.

Assume $\msRCo \nproves A \gives B$. Consider a finite adequate set
$\Phi$ containing $\{A,B\}$. Let
$\fk{A} = (W,\{R_\alpha\}_{\alpha \leq \omega},\val{\cdot})$ be a Kripke
model satisfying the conditions of~\Cref{th:rcsemantics} such that, for
some node $x \in W$, $\fk{A},x \models A$, yet $\fk{A},x \not\models B$.

As in the case of $\msGLP$, one can again assume that $\fk{A}$ is rooted
(see~\citep{beklemishev:2014}). Now one proceeds with the Solovay-type
construction similarly as for $\msGLP$. That is, one identifies the set
$W$ with a finite set of natural numbers $\{1,\ldots,N\}$ so that $1$ is
the root. One attaches a new root $0$ to $\fk{A}$ by stipulating that
$0R_0 x$, for all $x \in W$. The valuation of the variables at the new
root $0$ will be the same as in node $1$; abusing notation, let us call
the resulting model $\fk{A}$ as well. It is easy to check that $\fk{A}$
still satisfies the properties of~\Cref{th:rcsemantics} and that
$\fk{A},0 \models A$, but $\fk{A}, 0 \not\models B$. We assume that the
relation $x \in \val{C}$ (where $C \in \Phi$ is a positive formula) and the
relations $R_\alpha$ are naturally arithmetized by bounded formulas.

In the following, we shall denote by $\Prf_n(\alpha,y)$ an arithmetical
formula (of arithmetical complexity $\Delta_{n+1}$) expressing that
``\emph{$y$ is a proof of a formula $\alpha$ from the axioms of $\PA$
  and all true $\Pi_n$-sentences}''---recall that $\PA \proves \nbox{n}_\PA(\alpha) \lequiv \exists y\,\Prf_n(\alpha,y)$.
We again assume that each provable formula has arbitrarily long proofs
and that this holds provably in $\PA$.

Recall from the arithmetical completeness proof of $\msGLP$ that, if
$G(x,y)$ codes a function $g \colon \omega \rightarrow W$ in $\PA$, then
the formula $\ell^G = x$ is an abbreviation of the formula
$\exists N_0 \forall n \geq N_0\; G(n, x)$, i.e., the formula which
expresses the fact that $g$ reaches a limit at $x$.\footnote{We will
  reuse here most of the notation from the arithmetical completeness
  proof of $\msGLP$ without further comment.}

There is a striking difference in the arithmetical completeness proof of
$\RC\omega$ that makes it substantially different from that of $\GLP$:
since the arithmetical complexity of the uniform reflection principle is
unbounded, finitely many Solovay-style functions do not suffice for
obtaining completeness. Instead, in~\citep{beklemishev:2014}, infinitely
many such functions of increasing arithmetical complexity are
employed.

We are now going to state the major technical lemmas
from~\citep{beklemishev:2014} which will allow us to deduce an
arithmetical completeness theorem for $\msRCo$. First, the following
lemma states basic properties of the used Solovay-style functions:

\begin{lemma}[\citep{beklemishev:2014}]
  \label{lem:rccomplhelp1}
  Let $M$ denote the maximum modality $m < \omega$ occurring in $\Phi$,
  and $0$ if there is no such $m$. There is an infinite sequence
  $h_0,h_1,\ldots$ of functions of type $\omega \rightarrow W$ that
  satisfy the following properties:
  \begin{enumerate}[label={\rm (\arabic*)}]
  \item Each $h_k$ is defined by a respective formula $H_k$ in $\PA$
    which is $\Delta_{k+1}$ in $\PA$;
    \item the function $\varphi \colon k \longmapsto \code{H_k}$ is
      primitive recursive;
    \item for each $h_k$, we have that $h_k(x) = y$ if and only if, either
    \begin{itemize}
      \item $x = y = 0$, or
      \item $h_i(n) \neq h_i(n + 1) = y$, for some $i < k$, or
      \item $\exists m \geq \max\{M,k\}\,\Prf_k(\code{\ell^{H_m} \neq y},n)$ and $h_k(n) R_k y$ or $h_k(n) R_\omega y$, or
      \item $y = h_k(n)$.
    \end{itemize}
  \end{enumerate}
\end{lemma}

In the following, we fix such a sequence $h_0,h_1,\ldots$ of functions
with the properties as stated in~\Cref{lem:rccomplhelp1}
above. Informally speaking, the behavior of the functions $h_k$ in
comparison to those employed for $\msGLP$ can be described as follows
(see~\citep{beklemishev:2014}):
\begin{itemize}
\item The functions with lower index have higher priority in the sense
  that, whenever $h_m$ makes a move (i.e., if it changes its position to
  a new world from $W$), then $h_n$ will make the same move, for any
  $n > m$;
\item $h_k$ also reacts to proofs of limit statements of functions of
  lower priority, not only to those of itself;
\item $h_k$ is also allowed to move along $R_\omega$-edges.
\end{itemize}

\begin{lemma}[\citep{beklemishev:2014}]
  \label{lem:ellrcprop}
  For each $n,m$, provably in $\PA$,
  \begin{enumerate}[label={\rm (\arabic*)}]
    \item $\exists! z \in W\, \ell^{H_n} = z$;\label{lem:ellrcprop:1}
    \item $\ell^{H_n}R_{n+1}\ell^{H_{n+1}}$ or $\ell^{H_n}R_\omega\ell^{H_{n+1}}$ or $\ell^{H_n} = \ell^{H_{n+1}}$;
    \item if $m < n$ then $\ell^{H_m} = \ell^{H_n}$ or
      $\ell^{H_m}R_\alpha \ell^{H_n}$, for some
      $\alpha \in (m,n]\cup \{\omega\}$.\label{lem:ellrcprop:3}
  \end{enumerate}
\end{lemma}
The first item of~\Cref{lem:ellrcprop} states that every function
provably reaches a unique limit. The second item states that the limit
of $h_{n+1}$ is (provably) reachable from the limit of $h_n$ either via
an $R_{n+1}$-arc or an $R_\omega$-arc. The third item can be obtained
from the second one via an (external) induction on $n$.

For all $n < \omega$, we define an arithmetical formula $L_n(a)$ as
follows:
\begin{align*}
  L_n(a) \coloneqq \begin{cases}
    \exists x\, h_n(x) = a, & \text{if $n = 0$,}\\
    \exists x\,(h_n(x) = a \land \forall z \geq x\ h_{n-1}(z) = h_{n-1}(x)), & \text{otherwise.}
 \end{cases}
\end{align*}
Notice that $L_n(a)$ is expressible by a $\Sigma_{n+1}$-formula. As in
the arithmetical completeness proof for $\msGLP$, let $R^\ast_k(x)$
denote the set
$\{y \in W \mid \exists \alpha \geq k\colon x R_\alpha y\}$.

\begin{lemma}[\citep{beklemishev:2014}]
  \label{lem:ellmove}
  Let $k \geq n$ and $a \coloneqq \ell^{H_n}$. Then, provably in $\PA$,
  $L_n(a)$ implies that $\ell^{H_k} \in R_n^\ast(a) \cup \{a\}$.
\end{lemma}

Intuitively, \Cref{lem:ellmove} states that,
assuming $L_n(a)$ where $a$ is the limit of $h_n$, the limit of the
function $h_k$ is (provably) either $a$ or some point that is reachable
via a path from $a$ that consists of arcs $R_\alpha$, where
$n \leq \alpha \leq \omega$. This is because, due to the assumption
$L_n(a)$, $h_n$ can move only along such edges from $a$ onward.

The formulas $L_n(a)$ will be important for us to extract an
arithmetical realization that is typed. This is due to the following
lemma:

\begin{lemma}
  \label{lem:rctype}
  For all $n < \omega$ and all variables $p$ of sort $k \leq n$,
  provably in $\PA$,
  \begin{align*}
    \ell^{H_n} \in \val{p} \iff \forall w \in W \setminus \val{p}\, \lneg L_k(w).
  \end{align*}
\end{lemma}
\begin{proof}
  We reason in $\PA$ as follows. For the direction from left to right,
  suppose $\ell^{H_n} \in \val{p}$ and suppose to the contrary that
  there is a $w \in W$ and an $x$ such that $w \not\in \val{p}$ and
  $L_k(x, w)$. By strong persistence, we know that $v \not\in \val{p}$
  for all $v \in R_k^\ast(w)$. Since $k \leq n$, \Cref{lem:ellmove}
  gives us $\ell^{H_n} \in R_k^\ast(w) \cup \{w\}$, whence
  $\ell^{H_n} \in W \setminus \val{p}$. This contradicts the uniqueness of
  $\ell^{H_n}$ (that is, \cref{lem:ellrcprop:1} of~\Cref{lem:ellrcprop}).

  For the other direction, suppose (in $\PA$) that
  $\forall w \in W \setminus \val{p}\, \lneg L_k(w)$ and assume
  $\ell^{H_n} \neq x$ for all $x \in \val{p}$. By~\cref{lem:ellrcprop:1}
  of~\Cref{lem:ellrcprop}, it follows that
  $\ell^{H_n} \in W \setminus \val{p}$. Let $w \in W \setminus \val{p}$;
  we first prove that $\ell^{H_k} \neq w$. In case $k = 0$, by
  $\lneg L_k(w)$, we infer $\forall x\,h_k(x) \neq w$ and thus
  $\ell^{H_k} \neq w$. Suppose now that $k > 0$. Then $\lneg L_k(w)$ is
  equivalent to
  $\forall x\,(h_k(x) \neq w \lor \exists z \geq x\ h_{k-1}(z) \neq
  h_{k-1}(x))$. We claim that there are arbitrarily large $x$ such that
  $h_k(x) \neq w$. Indeed, suppose there is an $x_0$ such that
  $\forall y \geq x_0\ h_k(y) = w$. By $\lneg L_k(w)$, we infer that
  $\exists z \geq x_0\ h_{k-1}(z) \neq h_{k-1}(x_0)$, whence it follows
  that there is a $y_0 \geq x_0$ such that
  $h_{k-1}(x_0) = h_{k-1}(y_0) \neq h_{k-1}(y_0+1)$. By the definition
  of $h_k$, this implies $w = h_k(y_0 + 1) = h_{k-1}(y_0+1)$. Using the
  assumption $\lneg L_k(w)$ again, we infer that
  $\exists z \geq y_0 + 1\ h_{k-1}(z) \neq h_{k-1}(y_0 + 1)$. Thus,
  there is a $y_1 \geq y_0 + 1$ such that
  $w = h_{k-1}(y_0 + 1) = h_{k-1}(y_1) \neq h_{k - 1}(y_1 + 1)$. By the
  definition of $h_k$, this again implies
  $h_k(y_1 + 1) = h_{k-1}(y_1 + 1)$. Notice that
  $y_1 + 1 > y_1 \geq y_0 + 1 > y_0 \geq x_0$ and
  $h_k(x_0) = h_k(y_0 + 1) = h_k(y_1) = w$, but certainly
  $h_k(y_1 + 1) \neq w$. This contradicts the fact that
  $\forall y \geq x_0\ h_k(y) = w$. Thus, $\ell^{H_k}$ cannot reach its
  limit at $w$. It remains to observe that this entails
  $\ell^{H_k} \in \val{p}$ by~\cref{lem:ellrcprop:1}
  of~\Cref{lem:ellrcprop} and thus we infer $\ell^{H_n} \neq \ell^{H_k}$
  (recall that we have $\ell^{H_n} \in W \setminus \val{p}$). However,
  this means that $n > k$ and so, by~\cref{lem:ellrcprop:3}
  of~\Cref{lem:ellrcprop}, this implies that
  $\ell^{H_k}R_\alpha \ell^{H_n}$, for some
  $\alpha \in (k,n] \cup \{\omega\}$. This contradicts the property of
  $\fk{A}$ being strongly persistent, since $\ell^{H_k} \in \val{p}$ but
  $\ell^{H_n} \in W \setminus \val{p}$.
\end{proof}

We shall now define an appropriate arithmetical realization. Let
$\{\varphi_i : i \in I\}$ be a primitive recursive set of formulas. We
will denote by $[\varphi_i : i\in I]$ a numeration that numerates the
theory $\PA + \{\varphi_i : i \in I\}$. Using this notation, we now
define an arithmetical realization $f$ as follows:
\begin{align*}
  f(p) \coloneqq [\ell^{H_n} \in \val{p} : n \geq M]. 
\end{align*}
Notice that the formula $\ell^{H_n} \in \val{p}$ can indeed be
constructed primitive recursively from the parameter $n$, since the
function $\varphi \colon k \longmapsto \code{H_k}$ is primitive
recursive according to~\Cref{lem:rccomplhelp1}.

The following lemma states that, for $\sort{p} = k < \omega$, the
numeration $f(p)$ is indeed a $\Pi_{k+1}$-axiomatized extension of
$\PA$.

\begin{lemma}
  For each variable $p$ of sort $k < \omega$, $f(p)$ numerates a
  $\Pi_{k+1}$-axiomatized extension of $\PA$.
\end{lemma}
\begin{proof}
  Let $n \geq M$ and consider the sentence $\ell^{H_n} \in \val{p}$. If
  $k \leq n$, then by~\Cref{lem:rctype}, provably in $\PA$,
  \begin{align*}
    \ell^{H_n} \in \val{p} \iff \forall w \in W \setminus \val{p}\,\lneg L_k(w) \iff \bigwedge_{\mathclap{w \in W \setminus \val{p}}} \lneg L_k(\ol{w}).
  \end{align*}
  Recall that $L_k(x)$ is $\Sigma_{k+1}$ in $\PA$, whence it follows
  that $\lneg L_k(\ol{w})$ is $\Pi_{k+1}$ in $\PA$ and thus so is
  $\ell^{H_n} \in \val{p}$.

  For the case $k > n$, recall the very definition of
  $\ell^{H_n} \in \val{p}$ is
  $\bigvee_{x \in \val{p}} \ell^{H_n} = \ol{x}$ and the definition of
  $\ell^{H_n} = \ol{x}$ reads
  $\exists N_0\,\forall z > N_0\,H_n(z,\ol{x})$. By virtue
  of~\Cref{lem:rccomplhelp1}, $H_n(x,y)$ is $\Delta_{n+1}$ in $\PA$,
  whence it follows that $\ell^{H_n} = \ol{x}$ is $\Sigma_{n + 2}$ in
  $\PA$ and thus $\ell^{H_n} \neq \ol{x}$ is $\Pi_{n+2}$ in
  $\PA$. Observe that, by~\cref{lem:ellrcprop:1}
  of~\Cref{lem:ellrcprop}, provably in $\PA$,
  \begin{align*}
    \ell^{H_n} \in \val{p} \iff \bigwedge \{\ell \neq \ol{x} \mid x \in W \setminus \val{p}\}.
  \end{align*}
  Thus, $\ell^{H_n} \in \val{p}$ is $\Pi_{k+1}$ in $\PA$, since
  $k + 1 \geq n + 2$ by assumption.
\end{proof}

It follows that $f$ is actually a typed arithmetical
realization as desired. We can now proceed along the lines
of~\citep{beklemishev:2014} and quote some more technical lemmas that
will allow us to conclude the arithmetical completeness proof for
$\msRCo$:

\begin{lemma}[\citep{beklemishev:2014}]
  \label{lem:complhelper2}
  For any formula $C \in \Phi$,
  \begin{enumerate}[label={\rm(\arabic*)}]
    \item $[\ell^{H_n} \in \val{C} : n \geq M] \gives_\PA f(C)$;
    \item
      $\ul{\ell^{H_0} \neq 0} \lor f(C) \gives_\PA [\ell^{H_n} \in
      \val{C} : n \geq M]$.
  \end{enumerate}
\end{lemma}

\begin{lemma}[\citep{beklemishev:2014}]
  \label{lem:complhelper4}
  For all $n \geq 0$, $\mathbb{N} \models \ell^{H_n} = 0$.
\end{lemma}

Intuitively,~\Cref{lem:complhelper2} can be seen as a counterpart to the
``commutation lemma'' in the arithmetical completeness proof of $\msGLP$
(\Cref{lem:commutation}), while~\Cref{lem:complhelper4} simply states
that, in the standard model, all functions reach their limit 
at $0$.

Now we can conclude the proof of \Cref{arithmet-compl} in accordance
with~\citep{beklemishev:2014} as follows. Recall that we have
$\fk{A},1 \models A$ but $\fk{A}, 1 \not\models B$. Let $\sigma$ be the
numeration $[\ell^{H_n} = \ol{1} : n \geq M]$ and let $S$ be the theory
numerated by $\sigma$. By~\Cref{lem:complhelper2}, we know that
\begin{align*}
  \sigma &\gives_\PA [\ell^{H_n} \in \val{A} : n \geq M]\\
         &\gives_\PA f(A).
\end{align*}
By~\Cref{lem:complhelper2}, we also have
\begin{align*}
  \ul{\ell^{H_0} \neq 0} \lor f(B) &\gives_\PA [\ell^{H_n} \in \val{B} : n \geq M] \\
   &\gives_\PA [\ell^{H_n} \neq \ol{1} : n \geq M].
\end{align*}
Now if we had $f(A) \gives f(B)$, then
$S \proves \ell^{H_M} \neq \ol{1}$ and so $S$ would be inconsistent. One
can easily show that
$\PA \proves \ell^{H_n} = \ol{1} \limpl \ell^{H_m} = \ol{1}$, for all
$m \leq n$. Thus, there is a $\PA$-proof of $\ell^{H_n} \neq \ol{1}$,
for some $n \geq M$ (otherwise, $\PA \proves S$ and so $\PA$ would be
inconsistent too). But this means that $h_0$ must eventually take a
value different from $0$ by its definition. This is, however, impossible
due to~\Cref{lem:complhelper4}.

\section{Conclusion}
\label{sec:5}

We have studied a many-sorted fragment of $\msGLP$ where propositional
variables are assigned sorts $\alpha \leq \omega$. The logic $\msGLP$
admits a more fine-grained arithmetical interpretation than standard
$\GLP$: variables of finite sort $n < \omega$ range over
$\Pi_{n+1}$-sentences of the arithmetical hierarchy, while those of sort
$\omega$ range over arbitrary sentences. The inclusion of sorts in the
modal languages naturally corresponds, in the realm of modal logics, to
the notion of stratification of graded provability algebras in the
algebraic world. We showed that $\msGLP$ is arithmetically complete by
exploiting an existing construction for $\GLP$. Moreover, we reduced
$\msGLP$ to $\GLP$ and thereby transferred results from $\GLP$ to
$\msGLP$ like Craig interpolation and \PSpace\ decidability. We studied
variants of $\msGLP$ that restrict the use of sorts. A positive variant
of $\msGLP$, denoted $\msRCo$, was introduced which allows for an even
richer arithmetical interpretation due to the fact that variables are
permitted to range over arithmetical theories rather than single
sentences. This arithmetical interpretation allows the introduction of
an additional modality $\ndia{\omega}$ which is not present in $\msGLP$,
and which corresponds to the full uniform reflection principle. We
showed that $\msRCo$ is arithmetically complete by again relying on an
existing construction for its one-sorted counterpart.

\medskip
\noindent
\textbf{Funding.} This work was supported by the Austrian Science Fund
(FWF) [Y698 to G.B., W1255-N23 to H.T.]; by the Austrian Academy of Sciences [DOC Fellowship to G.B.]; and by the
Russian Foundation for Basic Research [15-01-09218 to
L.D.B.].

\medskip
\noindent
\textbf{Acknowledgments.} The authors would like to thank the anonymous
referees who provided useful comments for improving this paper.

  \bibliographystyle{plainnat}
  \bibliography{references}

\begin{thebibliography}{27}
\providecommand{\natexlab}[1]{#1}
\providecommand{\url}[1]{\texttt{#1}}
\expandafter\ifx\csname urlstyle\endcsname\relax
  \providecommand{\doi}[1]{doi: #1}\else
  \providecommand{\doi}{doi: \begingroup \urlstyle{rm}\Url}\fi

\bibitem[Ardeshir and Mojtahedi(2015)]{ardeshir:2015}
Mohammad Ardeshir and S.~Mojtaba Mojtahedi.
\newblock Reduction of provability logics to {$\Sigma_1$}-provability logics.
\newblock \emph{Logic Journal of the IGPL}, 23\penalty0 (5):\penalty0 842--847,
  2015.

\bibitem[Artemov and Beklemishev(2004)]{artemovbekl:2004}
Sergei~N. Artemov and Lev~D. Beklemishev.
\newblock {Provability Logic}.
\newblock In \emph{Handbook of Philosophical Logic, 2nd ed.}, pages 229--403.
  Kluwer, 2004.

\bibitem[Beklemishev(2004)]{beklemishev:2004}
Lev~D. Beklemishev.
\newblock {Provability algebras and proof-theoretic ordinals, I}.
\newblock \emph{Annals of Pure and Applied Logic}, 128\penalty0 (1-3):\penalty0
  103--123, 2004.

\bibitem[Beklemishev(2005)]{beklemishev:2005}
Lev~D. Beklemishev.
\newblock Reflection principles and provability algebras in formal arithmetic.
\newblock \emph{{Russian Mathematical Surveys}}, 60\penalty0 (2):\penalty0
  197--268, 2005.
\newblock Russian original: \emph{Uspekhi Matematicheskikh Nauk}, 60(2): 3--78,
  2005.

\bibitem[Beklemishev(2006)]{beklemishev:2006}
Lev~D. Beklemishev.
\newblock {The Worm Principle}.
\newblock In Z.~Chatzidakis, P.~Koepke, and W.~Pohlers, editors, \emph{Lecture
  Notes in Logic 27.~Logic Colloquium '02}, pages 75--95. AK Peters, 2006.
\newblock Preprint: Logic Group Preprint Series 219, Utrecht Univ., March 2003.

\bibitem[Beklemishev(2010{\natexlab{a}})]{beklemishev:2007a}
Lev~D. Beklemishev.
\newblock On the {Craig} interpolation and the fixed point properties of {GLP}.
\newblock In S.~Feferman et~al., editor, \emph{Proofs, Categories and
  Computations. Essays in honor of G. Mints}, Tributes, pages 49--60. College
  Publications, London, 2010{\natexlab{a}}.

\bibitem[Beklemishev(2010{\natexlab{b}})]{beklemishev:2010}
Lev~D. Beklemishev.
\newblock Kripke semantics for provability logic {GLP}.
\newblock \emph{Annals of Pure and Applied Logic}, 161\penalty0 (6):\penalty0
  756--774, 2010{\natexlab{b}}.

\bibitem[Beklemishev(2011)]{beklemishev:2011}
Lev~D. Beklemishev.
\newblock {A simplified proof of arithmetical completeness theorem for
  provability logic GLP}.
\newblock \emph{Proceedings of the Steklov Institute of Mathematics},
  274\penalty0 (1):\penalty0 25--33, 2011.

\bibitem[Beklemishev(2012)]{beklemishev:2012}
Lev~D. Beklemishev.
\newblock Calibrating provability logic: from modal logic to reflection
  calculus.
\newblock In T.~Bolander, T.~Bra\"{u}ner, S.~Ghilardi, and L.~Moss, editors,
  \emph{{Advances in Modal Logic, v.~9}}, pages 89--94. College Publications,
  London, 2012.

\bibitem[Beklemishev(2014)]{beklemishev:2014}
Lev~D. Beklemishev.
\newblock {Positive provability logic for uniform reflection principles}.
\newblock \emph{{Annals of Pure and Applied Logic}}, 165\penalty0 (1):\penalty0
  82--105, 2014.

\bibitem[Beklemishev(2017{\natexlab{a}})]{Bek17}
Lev~D. Beklemishev.
\newblock On the reduction property for {GLP-algebras}.
\newblock \emph{Doklady: Mathematics}, 95\penalty0 (1):\penalty0 50--54,
  2017{\natexlab{a}}.

\bibitem[Beklemishev(2017{\natexlab{b}})]{beklemishev:2017}
Lev~D. Beklemishev.
\newblock On the reflection calculus with partial conservativity operators.
\newblock In \emph{Logic, Language, Information, and Computation - 24th
  International Workshop, WoLLIC 2017, London, UK, July 18-21, 2017,
  Proceedings}, pages 48--67, 2017{\natexlab{b}}.

\bibitem[Beklemishev et~al.(2014)Beklemishev, Fern\'a{}ndez-Duque, and
  Joosten]{beklemishevjoost:2014}
Lev~D. Beklemishev, David Fern\'a{}ndez-Duque, and Joost~J. Joosten.
\newblock On provability logics with linearly ordered modalities.
\newblock \emph{Studia Logica}, 102\penalty0 (3):\penalty0 541--566, 2014.

\bibitem[Bezhanishvili and de~Jongh(2006)]{dejongh:2006}
Nick Bezhanishvili and Dick de~Jongh.
\newblock Intuitionistic logic.
\newblock Technical report, Institute for Logic, Language and Computation,
  University of Amsterdam, 2006.

\bibitem[Boolos(1993)]{boolos:1995}
George~S. Boolos.
\newblock \emph{The Logic of Provability}.
\newblock Cambridge University Press, 1993.

\bibitem[Dashkov(2012)]{dashkov:2012}
Evgeny~V. Dashkov.
\newblock {On the positive fragment of the polymodal provability logic GLP}.
\newblock \emph{{Mathematical Notes}}, 91\penalty0 (3):\penalty0 318--333,
  2012.
\newblock Original Russian text in: \emph{Matematicheskie Zametki},
  91:(3):331--336, 2012.

\bibitem[Dzhaparidze(1988)]{japaridze:1988}
Giorgie Dzhaparidze.
\newblock {The polymodal logic of provability}.
\newblock In \emph{Intensional Logics and Logical Structure of Theories:
  Material from the fourth Soviet-Finnish Symposium on Logic, Telavi, May
  20-24, 1985, Metsniereba, Tbilisi}, pages 16--48, 1988.
\newblock In Russian.

\bibitem[Dzhaparidze(1994)]{Dzh94}
Giorgie Dzhaparidze.
\newblock The logic of arithmetical hierarchy.
\newblock \emph{Annals of Pure and Applied Logic}, 66\penalty0 (2):\penalty0
  89--112, 1994.

\bibitem[Feferman(1996)]{Fefbug}
Solomon Feferman.
\newblock Three conceptual problems that bug me.
\newblock {Lecture text for 7-th Scandinavian Logic Symposium,
  https://math.stanford.edu/$\sim$feferman/papers/conceptualprobs.pdf}, 1996.

\bibitem[Fern\'andez-Duque and Joosten(2013)]{joostenf:2013}
David Fern\'andez-Duque and Joost~J. Joosten.
\newblock The omega-rule interpretation of transfinite provability logic.
\newblock \emph{CoRR}, abs/1302.5393, 2013.

\bibitem[Ignatiev(1993)]{ignatiev:1993}
Konstantin~N. Ignatiev.
\newblock {On Strong Provability Predicates and the Associated Modal Logics}.
\newblock \emph{The Journal of Symbolic Logic}, 58\penalty0 (1):\penalty0
  249--290, 03 1993.

\bibitem[Kreisel(1977)]{kreisel:1977}
Georg Kreisel.
\newblock {Wie die Beweistheorie zu ihren Ordinalzahlen kam und kommt}.
\newblock \emph{Jahresbericht der DMV}, 78\penalty0 (4):\penalty0 177--223,
  1977.

\bibitem[Shapirovsky(2008)]{shapirovsky:2008}
Ilya Shapirovsky.
\newblock {PSPACE-decidability of Japaridze's polymodal logic}.
\newblock In \emph{Advances in Modal Logic}, volume~7, pages 289--304, 2008.

\bibitem[Solovay(1976)]{solovay:1976}
Robert~M. Solovay.
\newblock Provability interpretations of modal logic.
\newblock \emph{Israel Journal of Mathematics}, 25\penalty0 (3-4):\penalty0
  287--304, 1976.

\bibitem[Visser(1981)]{visser:1981}
Albert Visser.
\newblock \emph{Aspects of diagonalization and provability}.
\newblock PhD thesis, Utrecht University, 1981.

\bibitem[Visser(1984)]{Vis84}
Albert Visser.
\newblock The provability logics of recursively enumerable theories extending
  {P}eano {A}rithmetic at arbitrary theories extending {P}eano {A}rithmetic.
\newblock \emph{Journal of Philosophical Logic}, 13:\penalty0 97--113, 1984.

\bibitem[Visser(2002)]{Vis02a}
Albert Visser.
\newblock Substitutions of {$\Sigma_1^0$}-sentences: {E}xplorations between
  intuitionistic propositional logic and intuitionistic arithmetic.
\newblock \emph{Annals of Pure and Applied Logic}, 114\penalty0
  (1--3):\penalty0 227--271, 2002.

\end{thebibliography}
\end{document}